\documentclass{amsart}
\numberwithin{equation}{subsection}

\usepackage{bbm}
\usepackage{comment}
\usepackage{verbatim}
\usepackage{mathrsfs}
\usepackage{mathtools}
\usepackage{latexsym}
\usepackage{epsfig}
\usepackage{amsmath}
\usepackage[usenames,dvipsnames]{xcolor}
\usepackage{graphics}
\usepackage[utf8]{inputenc}
\usepackage{amssymb}
\usepackage{amsthm}
\usepackage[alphabetic]{amsrefs}
\usepackage{amsopn}
\usepackage{amscd}
\usepackage[all,knot]{xy}
\xyoption{all}
\usepackage{rotating}
\usepackage{tikz}
\usetikzlibrary{calc}
\usepgflibrary{shapes.geometric}
\usepgflibrary{shapes.misc}
\usetikzlibrary{positioning}
\usetikzlibrary{decorations}
\usetikzlibrary{decorations.pathreplacing}
\usepackage{hyperref}
\usepackage{tikz}
\usepackage{tikz-cd}

\theoremstyle{plain}
\newtheorem{thm}{Theorem}[subsection]
\newtheorem{lem}[thm]{Lemma}
\newtheorem{prop}[thm]{Proposition}
\newtheorem{cor}[thm]{Corollary}

\newtheorem*{thm*}{Theorem}
\newtheorem*{lem*}{Lemma}
\newtheorem*{prop*}{Proposition}
\newtheorem*{cor*}{Corollary}

\theoremstyle{definition}
\newtheorem{defn}[thm]{Definition}
\newtheorem*{defn*}{Definition}

\newtheorem{ex}[thm]{Example}
\newtheorem{con}[thm]{Construction}{}
\newtheorem{rem}[thm]{Remark}
\newtheorem*{rem*}{Remark}

{}
{}
{}
\newtheorem*{ack}{Acknowledgements}{}
\newtheorem{qn}[thm]{Question}{}

\theoremstyle{remark}
{}
{}
{}

\def\to{\longrightarrow} 

\def\Set{\mathsf{Set}}

\def\AA{\mathbb{A}}

\def\CC{\mathbb{C}}

\def\FF{\mathbb{F}}

\def\PP{\mathbb{P}}

\def\ZZ{\mathbb{Z}}
\DeclareUnicodeCharacter{221E}{$\infty$}

\def\sfA{\mathsf{A}}
\def\sfB{\mathsf{B}}

\def\sfD{\mathsf{D}}
\def\sfE{\mathsf{E}}

\def\sfI{\mathsf{I}}
\def\sfJ{\mathsf{J}}
\def\sfK{\mathsf{K}}
\def\sfL{\mathsf{L}}
\def\sfM{\mathsf{M}}
\def\sfN{\mathsf{N}}

\def\sfP{\mathsf{P}}
\def\sfQ{\mathsf{Q}}
\def\sfR{\mathsf{R}}
\def\sfS{\mathsf{S}}
\def\sfT{\mathsf{T}}

\def\sfX{\mathsf{X}}

\def\mcF{\mathcal{F}}

\def\mcO{\mathcal{O}}
\def\mcP{\mathcal{P}}

\def\mcS{\mathcal{S}}

\def\mfm{\mathfrak{m}}

\def\mfp{\mathfrak{p}}
\def\mfq{\mathfrak{q}}

\def\op{\mathrm{op}}
\def\b{\mathrm{b}}

\def\unit{\mathbf{1}}

\DeclareMathOperator{\add}{add}

\DeclareMathOperator{\id}{id}

\DeclareMathOperator{\Ob}{Ob}

\DeclareMathOperator{\Spec}{Spec}

\DeclareMathOperator{\Spc}{Spc}

\DeclareMathOperator{\supp}{supp}

\DeclareMathOperator{\thick}{thick}
\DeclareMathOperator{\Thick}{Thick}

\DeclareMathOperator{\Ind}{Ind}

\DeclareMathOperator{\Zar}{Zar}

\DeclareMathOperator{\pt}{pt}
\DeclareMathOperator{\spt}{spt}
\DeclareMathOperator{\fSpcnt}{fSpcnt}
\DeclareMathOperator{\Spcnt}{Spcnt}
\DeclareMathOperator{\MSpec}{MSpc}

\def\Frm{\mathsf{Frm}}
\def\CohFrm{\mathsf{CohFrm}}
\def\SFrm{\mathsf{SFrm}}
\def\CLat{\mathsf{CLat}}
\def\CjSLat{\mathsf{CjSLat}}
\def\Tcat{\mathsf{tcat}}
\def\ttcat{\mathsf{ttcat}}
\def\rttcat{\mathsf{rttcat}}
\def\hTcat{\mathsf{tcat}_\wedge}
\def\Sob{\mathsf{Sob}}
\def\Sierp{\mathbf{2}}

\def\lefta{d}
\def\lefter{\Omega}

\definecolor{internationalkleinblue}{rgb}{0.0, 0.18, 0.65}

\title{Approximating triangulated categories by spaces}

\author{Sira Gratz}
\thanks{The first author was supported by Villum Fonden and by EPSRC (Grant Number EP/V038672/1).}
\address{Sira Gratz, Aarhus University, Department of Mathematics, Ny Munkegade 118, bldg. 1530
DK-8000 Aarhus C, Denmark
}
\email{sira@math.au.dk}
\urladdr{https://sites.google.com/view/siragratz}

\author{Greg Stevenson}
\address{Greg Stevenson, Aarhus University, Department of Mathematics, Ny Munkegade 118, bldg. 1530
DK-8000 Aarhus C, Denmark
}
\email{greg@math.au.dk}
\urladdr{https://sites.google.com/view/gregstevenson}

\dedicatory{To Henning Krause on the occasion of his 60th birthday}


\begin{document}

\begin{abstract}
\noindent We initiate a systematic study of lattices of thick subcategories for arbitrary essentially small triangulated categories. To this end we give several examples illustrating the various properties these lattices may, or may not, have and show that as soon as a lattice of thick subcategories is distributive it is automatically a spatial frame. We then construct two non-commutative spectra, one functorial and one with restricted functoriality, that give universal approximations of these lattices by spaces.
\end{abstract}

\maketitle

\setcounter{tocdepth}{1}
\tableofcontents



\section{Introduction}

What do the bounded derived categories of the projective plane $\PP^2_\CC$, the $3$-Kronecker quiver $K_3$, and the group algebra of the symmetric group on $4$ letters over $\overline{\FF}_2$ have in common? One answer is that we don't really understand any of them. For instance, we have very little information on what the possible localizations of these categories are. We don't understand what the possible kernels of such localizations, i.e.\ the thick subcategories, are\textemdash{}let alone what the resulting quotient looks like. Even in these small examples there is much to be done; it was only recently shown by Pirozhkov \cite{P2phantoms} that there are no phantom subcategories, i.e.\ thick subcategories not detected by additive invariants, for $\PP^2$ and the corresponding question is still open for $3$-Kronecker.

What we understand quite well are the \emph{monoidal} localizations. Both $\sfD^\mathrm{b}(\PP^2)$ and $\sfD^\mathrm{b}(\overline{\FF}_2S_4)$ are equipped with compatible symmetric monoidal structures and the possible monoidal localizations were classified by Thomason \cite{Thomclass} in the former case and Benson, Carlson, and Rickard in the latter \cite{BCR}. The study of these monoidal localizations comprises tensor triangular geometry, affectionately known as tt-geometry, which was introduced by Balmer \cite{BaSpec} and has rapidly matured, aided by the plethora of tantalizing examples, into a very active field of study.

The idea is to functorially assign to every tensor triangulated category (tt-category) a space parameterizing thick ideals, i.e.\ thick subcategories closed under tensoring and which are radical. This creates opportunities to work geometrically and topologically to study the behaviour of tt-categories and to probe their objects in terms of an associated support theory. 

The secret sauce is \emph{distributivity} of the lattice of radical thick tensor ideals. Distributivity is the bare minimum for a lattice to be isomorphic to the lattice of open subsets of some space. One can interpret intersections of ideals in terms of the monoidal structure and so, as for radical ideals in a commutative ring, intersections of ideals distribute over sums of ideals. The insight is that the lattices in tt-geometry are always lattices of opens and so one can pass, without loss of information, from the lattice to the associated space\textemdash{}the spectrum.

The aim of this paper is to initiate a systematic and general study of lattices of thick subcategories of triangulated categories. To this end we make a contribution to the taxonomy of properties that such a lattice may, or may not, possess including distributivity and various weakenings thereof. This occupies Section~\ref{sec:shithappens}. Our first main observation is that for lattices of thick subcategories the bare minumum requirement, i.e.\ distributivity, suffices to be the lattice of opens of a topological space.

\begin{thm*}[Corollary~\ref{cor:spatial}]
Let $\sfK$ be an essentially small triangulated category with lattice of thick subcategories $\Thick(\sfK)$. If $\Thick(\sfK)$ is distributive then there is a sober topological space $X$, given by prime thick subcategories of $\sfK$, such that $\Thick(\sfK)$ is isomorphic to the lattice of open subsets of $X$.
\end{thm*}

However, distributivity does not hold in many examples. None of the projective plane, $3$-Kronecker, or $\overline{\FF}_2S_4$ have bounded derived categories with a distributive lattice of thick subcategories. Indeed, when there is some discrete combinatorial data involved in understanding objects, for example the twisting sheaves and their mutations on $\PP^2$ or the preprojectives and preinjectives for $3$-Kronecker, this seems to rule out distributivity. So, how are we to approach thick subcategories in these settings?

As a first step in this direction we construct, in a universal way, two spaces one can associate to \emph{any} essentially small triangulated category $\sfK$ in order to get a measure of its lattice of thick subcategories. In the absence of distributivity these spaces cannot faithfully reflect the structure of $\Thick(\sfK)$. The first one, which is functorial with respect to all exact functors, adds information and the second loses information unless the lattice is distributive.

Let us denote by $\Tcat$ the category of essentially small triangulated categories and exact functors, and by $\Sob$ the category of sober spaces and continuous maps.

\begin{thm*}[Section~\ref{ssec:universal}]
There is a functor $\fSpcnt\colon \Tcat^\op \to \Sob$, and an associated notion of support, such that for each $\sfK \in \Tcat$ the pair $(\fSpcnt \sfK, \supp)$ is the terminal support theory. 
\end{thm*}

One can think of $\fSpcnt \sfK$ as a coarse moduli space for thick subcategories of $\sfK$, where the topology reflects the various inclusion relations. The second space is, in some sense, a refinement of $\fSpcnt$ which is still free (in the technical sense), but subject to more constraints.

\begin{thm*}[Section~\ref{sec:pf}]
For any essentially small triangulated category $\sfK$ there exists a sober space $\Spcnt \sfK$ such that if $\Thick(\sfK)$ is distributive then $\mcO\Spcnt \sfK$, the lattice of open subsets of $\Spcnt \sfK$ is isomorphic to $\Thick(\sfK)$. This space comes with a canonical comparison map $\Spcnt \sfK \to \fSpcnt \sfK$.
\end{thm*}

These spaces are constructed and studied in Sections \ref{sec:ff} and \ref{sec:pf} respectively. The space $\Spcnt \sfK$, which we dub the `non-tensor spectrum' of $\sfK$, admits natural comparison maps to the Balmer spectrum when $\sfK$ is a rigid tt-category and to the noncommutative spectrum of \cite{nakano2019} when $\sfK$ is a monoidal triangulated category such that every 2-sided thick $\otimes$-ideal is semiprime. It also recovers the spectrum defined by Matsui \cite{matsui2019} in examples e.g.\ singularity categories of complete intersections.

However, caveat emptor: the non-tensor spectrum can be empty in non-trivial examples! In fact, one of the original motivations for this work was to illustrate that, without distributivity hypotheses, one can't develop a reasonable theory of supports on a space which is functorial and compatible with other theories. One seems to require some refined notion combining the continuous and discrete pieces of the classification. A tempting direction for future work is to consider spaces with some form of decoration e.g.\ those appearing in \cite{StevensonAntieau} and \cite{GratzZvonareva}.

We conclude the paper with two sections discussing examples and a small selection of open problems respectively.

\begin{ack}
We are grateful to David Ploog and Charalampos Verasdanis for helpful comments on a preliminary version of this work. We also thank the anonymous referee for their suggested improvements and for pointing out an inaccuracy.
\end{ack}



\section{Preliminaries on lattices}

In this section we recall various facts about lattices, spaces, and the duality between them that we will use throughout. We do not aim to give a complete overview, and recommend \cite{sketches} and \cite{stonespaces} as useful references.

\subsection{The players}

We begin by introducing various characters and the relations between them.

\begin{defn}\label{clat}
A \emph{lattice} $L$ is a poset which has a least upper bound and greatest lower bound for every finite non-empty set of elements, i.e.\ has finite non-empty joins and meets which we denote by $\vee$ and $\wedge$ respectively. We say $L$ is \emph{complete} if it admits meets of arbitrary sets of elements, and note that this implies that arbitrary joins also exist (equivalently $L$ is complete if it admits arbitrary joins and the existence of meets follows). We denote the top and bottom of a lattice, whenever they exist, by $1$ and $0$ respectively. Note that a complete lattice has both a top and a bottom.

There are two different categories of complete lattices that we will be concerned with. We denote by $\CjSLat$ the category of complete lattices (thought of as complete join semilattices) and poset maps which preserve arbitrary joins. We note that the bottom element of a lattice, being the empty join, is preserved by any such map, and that we do not require the top element to be preserved. We denote by $\CLat$ the category of complete lattices and poset maps which preserve arbitrary joins and finite meets, and hence both the top and bottom.
\end{defn}

For our purposes complete lattices are not sufficient. We ultimately want to work with lattices that arise as the poset of open subsets of a space. This leads us to the notion of spatial frames.

\begin{defn}\label{frm}
A \emph{frame} $F$ is a complete lattice in which finite meets distribute over arbitrary joins. We denote by $\Frm$ the category of frames with poset maps which preserve finite meets and arbitrary joins. 
\end{defn}

\begin{defn}
Let $L$ be a lattice. An element $l\in L$ is \emph{compact} (or \emph{finitely presented}) if, for any set of elements $\{m_i \mid i\in I\} \subseteq L$, we have
\[
l \leq \bigvee_{i\in I}m_i \text{ if and only if } \exists i_1,\ldots, i_n \text{ such that } l \leq \bigvee_{j=1}^n m_{i_j}.
\]
Viewing $L$ as a category this just says that $l$ is a finitely presented object, i.e.\ $L(l,-)$ commutes with filtered colimits.
\end{defn}

\begin{defn}\label{cohfrm}
A frame $F$ is \emph{coherent} if its finitely presented objects
form a bounded sublattice (in particular $1$ must be finitely presented) and generate $F$ under joins. 
\end{defn}

\begin{rem}
From the categorical perspective this is just asking that $F$ is locally finitely presented and the finitely presented objects of $F$ are closed under finite limits and colimits.
\end{rem}

There is an obvious fully faithful inclusion functor $\Frm \to \CLat$ and a (faithful) forgetful functor $\Frm \to \CjSLat$.

We denote by $\Sierp$ the unique frame with two elements $\{0 \leq 1\}$.

\begin{defn}\label{pt}
A \emph{point} of a complete lattice $L$ is a morphism $L \to \Sierp$ in $\CLat$. The set of points of $L$ is thus $\CLat(L,\Sierp)$. 
\end{defn}

Since $\Frm$ is a full subcategory of $\CLat$ this reduces to the usual notion for frames.

We take this opportunity to recall another definition which is intimately connected to the notion of points.

\begin{defn}\label{defn:meetprime}
An element $\mfp\neq 1$ of a lattice $L$ is called \emph{meet-prime} if whenever $l \wedge m \leq \mfp$ then we have $l 
\leq \mfp$ or $m \leq \mfp$.
\end{defn}

\begin{rem}
There is a bijection between points of $L$ and the set $\mcP$ of meet-prime elements of $L$. Explicitly, this is given by sending a meet-prime element $\mfp$ to the point $p$ defined by
\[
p(a) = \begin{cases}
0  & \text{ if } a\leq \mfp \\
1  & \text{ if } a\nleq \mfp.
\end{cases}
\]
The inverse is given by sending a point $p$ to the element
\[
\mfp = \bigvee\{ a \mid p(a)=0\}
\]
which is easily checked to be meet-prime.
\end{rem}

\begin{defn}\label{sfrm}
A frame $F$ is \emph{spatial} if it has enough points in the sense that whenever $a\nleq b$ in $F$ there exists a point $p\colon F\to \Sierp$ such that $p(a)=1$ and $p(b)=0$.

We denote by $\SFrm$ the full subcategory of $\Frm$ consisting of the spatial frames. 
\end{defn}

We now turn to the topological incarnation of spatial frames and the passage between spaces and frames.

\begin{defn}\label{Sob}
A topological space $X$ is \emph{sober} if every irreducible closed subset of $X$ has a unique generic point. We denote by $\Sob$ the category of Sober spaces and continuous maps.
\end{defn}

We recall that $\SFrm$ and $\Sob$ are dual, i.e.\ there are quasi-inverse equivalences
\begin{displaymath}
\xymatrix{
\SFrm^\op \ar[rr]<0.75ex>^-{\pt} \ar@{<-}[rr]<-0.75ex>_-{\mcO} && \Sob
}
\end{displaymath}
where for a frame $F$ 
\begin{displaymath}
\pt(F) = \Frm(F, \Sierp)
\end{displaymath}
with open subsets, indexed by $l\in F$, given by $U_l = \{p\colon F\to \Sierp \mid p(l) = 1\}$, and for a sober space $X$ we define $\mcO(X)$ to be the frame of open subsets of $X$, which is spatial (for any space $X$). This equivalence is known as Stone duality.

\begin{rem}
Sometimes $\pt(F)$ is also denoted by $\Spec(F)$ and called the spectrum of $F$ in analogy with algebraic geometry.
\end{rem}

The following lemma is an immediate consequence of this duality; the essential point is that the space with one point is a generator for the category of Sober spaces (in the sense of separating distinct morphisms, i.e. two distinct parallel morphisms must differ at some point).

\begin{lem}\label{lem:SAFT}
The category $\SFrm$ of spatial frames is cogenerated by $\Sierp$. Moreover, it is complete, locally small, and well-powered and so satisfies the hypotheses of the Special Adjoint Functor Theorem (as formulated in \cite{MacLane}*{Chapter~V.8}).
\end{lem}

Before moving on we briefly return to coherent frames as introduced in Definition~\ref{cohfrm}. A coherent frame is automatically spatial, as shown for instance in \cite{stonespaces}*{Theorem~II.3.4}. Thus, under Stone duality, the coherent frames correspond to a distinguished class of sober spaces. 

\begin{defn}\label{defn:spectral}
The essential image of the coherent frames under Stone duality consists of the \emph{spectral spaces} (sometimes also called coherent spaces for coherence with the frame terminology). These are the quasi-compact sober spaces such that the quasi-compact open subsets are closed under finite intersections and form a basis for the topology. 
\end{defn}

As shown by Hochster \cite{Hochster} the spectral spaces are precisely the spectra of commutative rings. Their relevance to us is that they appear naturally in the context of tt-geometry. There is a natural notion of duality for spectral spaces which will appear in the comparison of our work to the tt-framework.

\begin{defn}\label{defn:Hochsterdual}
Let $X$ be a spectral space. The \emph{Hochster dual} of $X$, denoted by $X^\vee$, has the same points as $X$ and the topology generated by taking the closed subsets of $X$ with quasi-compact open complement as a basis of open subsets for $X^\vee$. 
\end{defn}

\begin{rem}
Another point of view is given as follows: let $K(X)$ denote the lattice of quasi-compact open subsets of $X$, i.e.\ the compact objects of the coherent frame $\mcO(X)$. The dual $X^\vee$ corresponds to the space of points of $\Ind(K(X)^\op)$, the coherent frame given by the ind-completion (i.e.\ lattice of ideals) of the opposite lattice of $K(X)$.
\end{rem}


\subsection{Limits}\label{ssec:limits}

We now briefly discuss limits in the various categories we have introduced above. By \cite{sketches}*{Lemma~C.1.1.3} the forgetful functors to $\Set$ create limits for $\Frm$ and $\CjSLat$ (one defines the lattice operations pointwise on products) and the forgetful functor $\Frm \to \CjSLat$ has a left adjoint. One can easily check that the forgetful functor to $\Set$ also creates limits for $\CLat$.

\begin{lem}\label{lem:limclosed}
The full subcategory $\SFrm$ of $\Frm$ is closed under limits.
\end{lem}
\begin{proof}
It is enough to check that $\SFrm$ is closed under products and equalizers.

Suppose that $L_1$ and $L_2$ are spatial frames and we are given an equalizer diagram, taken in $\Frm$,
\begin{displaymath}
\xymatrix{
L\ar[r]^-e & L_1 \ar[r]<0.5ex>^-f \ar[r]<-0.5ex>_-g & L_2
}
\end{displaymath}
so that
\begin{displaymath}
L = \{l\in L_1 \mid f(l) = g(l)\}
\end{displaymath}
by the discussion above. If $l \nleq l'$ in $L$, then this is still true in $L_1$, and so there exists a point $p\colon L_1\to \Sierp$ separating these elements. Thus one can separate these elements of $L$ via the point $L \stackrel{e}{\to} L_1 \stackrel{p}{\to} \Sierp$.

Suppose $\{L_i \mid i\in I\}$ is a set of spatial frames with product (taken in $\Frm$) $L$ and $l\nleq l'$ in $L$. Then, there exists some $i\in I$ such that in the $i$th component $l_i \nleq l'_i$ and so, since $L_i$ is spatial, a point $p\colon L_i \to \Sierp$ separating $l_i$ and $l'_i$. Then precomposing with the canonical projection $L\to L_i$ gives a point of $L$ separating $l$ and $l'$.
\end{proof}


\subsection{The adjunctions}

Our approach to studying lattices of thick subcategories will be centred around the existence of two adjoint functors; the point of this section is to exhibit them. We can immediately deduce the existence of the first from general nonsense.

\begin{prop}\label{prop:adjoint1}
The forgetful functor $i\colon \SFrm \to \CjSLat$ has a left adjoint $\lefta$.
\end{prop}
\begin{proof}
We have seen in Lemma~\ref{lem:limclosed} that the inclusion $\SFrm \to \Frm$ preserves limits. By \cite{sketches}*{Lemma~C.1.1.3} the canonical functor $\Frm \to \CjSLat$ preserves limits, and thus the composite $i\colon \SFrm \to \CjSLat$ preserves limits. The category $\CjSLat$ is certainly locally small and so by Lemma~\ref{lem:SAFT} the Special Adjoint Functor Theorem applies to yield the left adjoint $\lefta$. 
\end{proof}

It is possible to give a concrete construction of $\lefta$. We sketch the construction here, in order that the examples we give later are not completely out of the blue, but do not provide full justifications.

\begin{con}\label{cons:left}
Let $L$ be a complete join semilattice. We define the collection of \emph{semipoints} of $L$ to be
\begin{displaymath}
\spt(L) = \CjSLat(L, \Sierp),
\end{displaymath}
i.e.\ the collection of join (but not necessarily meet) preserving maps to the frame $\Sierp$. As in the discussion of Stone duality (and as to come in Construction~\ref{cons:lefter}) we define for $l\in L$ a subset
\begin{displaymath}
U_l = \{p\in \spt(L) \mid p(l)=1\}
\end{displaymath}
and note that, for any collection of elements $\{l_\lambda \mid \lambda \in \Lambda\}$ of $L$, we have $U_{\vee_{\lambda} l_\lambda} = \cup_{\lambda} U_{l_\lambda}$. These subsets will, in general, not be closed under finite intersections (for example $\spt(L)$ need not be of this form), but we take the topology they generate (so the open subsets are unions of finite intersections of $U_l$'s) and endow $\spt(L)$ with the structure of a topological space.

The free spatial frame on $L$, $\lefta(L)$, is given by the frame of open subsets $\mcO\spt(L)$ of $\spt(L)$. (To be completely precise this should be $\mcO^\op\spt^\op$, to be precise about the action on maps, but we omit these op's throughout for readability). The unit of the adjunction $L \to \lefta(L)$ is given by the obvious assignment $l \mapsto U_l$, and for a frame $F$ the counit is given by
\begin{displaymath}
\varepsilon\colon \lefta(F) \to F \quad \quad \varepsilon(\bigcup_i (\bigcap_{j=1}^{n_i} U_{f_j})) = \bigvee_i (\bigwedge_{j=1}^{n_i} f_j).
\end{displaymath}
\end{con}

We also wish to consider the fully faithful functor $j\colon \SFrm \to \CLat$ which, as we shall see, is also a right adjoint. This follows, as above, by noting that limits in $\CLat$ are created by the forgetful functor. However, we wish to give a direct construction of the left adjoint which we denote by $\lefter \colon \CLat \to \SFrm$. 

\begin{con}\label{cons:lefter}
Given a complete lattice $L\in \CLat$ we can construct from it a space exactly as if it were a frame. Indeed, the space of points of $L$ is $\CLat(L, \Sierp)$ (note that this is taken in $\CLat$ and so maps are required to preserve all joins and finite meets), with open subsets the $U_l = \{p\colon L\to \Sierp \mid p(l)=1\}$. One checks easily (using that joins and meets are preserved as indicated) that these are the open subsets of a topology. As for a frame, we often denote this space by either $\pt(L)$ or $\Spec(L)$.

We define $\lefter L$ to be $\mcO\CLat(L,\Sierp)$, the spatial frame of open subsets of this space. Both $\mcO$ and $\CLat(-,\Sierp)$ are (contravariantly) functorial and so we get a functor $\lefter\colon \CLat \to \SFrm$.

Now let us construct the unit and counit of adjunction. The unit for $L\in \CLat$ as above is given by
\begin{displaymath}
L \stackrel{\eta_L}{\to} j\lefter L \quad \text{via} \quad l \mapsto U_l.
\end{displaymath}
On the other hand for a spatial frame $F$ we know that $\lefter jF$ is naturally isomorphic to the functor $F$, via the duality with $\Sob$. We thus take for the counit $\varepsilon\colon \lefter jF \to F$ the counit of the equivalence $\SFrm \cong \Sob^\op$.
\end{con}

\begin{prop}\label{prop:adjoint2}
The above construction determines an adjunction $\lefter \dashv j$ between $\SFrm$ and $\CLat$.
\end{prop}
\begin{proof}
We check the triangle identities, which is straightforward in this case. We start with the composite
\begin{displaymath}
\xymatrix{
\lefter \ar[r]^-{\lefter(\eta)} & \lefter j \lefter \ar[r]^-{\varepsilon_\lefter} & \lefter.
}
\end{displaymath}
Since $\lefter$ lands in $\SFrm$ this composite is the identity by Stone duality (using our choice of the counit). The other identity concerns the composite
\begin{displaymath}
\xymatrix{
j \ar[r]^-{\eta_j} & j \lefter j \ar[r]^-{j(\varepsilon)} & j,
}
\end{displaymath}
which can be identified with $j$ applied to the composite $\id_{\SFrm} \to \lefter j \to \id_{\SFrm}$ by using that $j$ is fully faithful. This latter composite is again the identity by Stone duality.
\end{proof}

Before continuing with our discussion let us make the following observation which will be useful in the sequel; we use the concepts recalled in Definition~\ref{defn:meetprime} and the remark following it. 

\begin{lem}\label{lem:sober}
Let $L$ be a complete lattice. The space $X = \CLat(L,\Sierp)$ is sober.
\end{lem}
\begin{proof}
The closed subsets of $X$ are precisely those of the form
\begin{displaymath}
Z_l = \{p\in \CLat(L,\Sierp) \mid p(l)=0\} \text{ for } l\in L.
\end{displaymath}
Under the bijection between points of $L$ and the set $\mcP$ of meet-prime elements we have $Z_l = \{\mfp \in \mcP \mid l\leq \mfp\}$. We note that, under this identification, it is clear that for $\mfq\in \mcP$ the subset $Z_\mfq$ is irreducible with generic point $\mfq$.

Suppose that $Z_l$ is irreducible, and consider $\mfq = \bigwedge_{\mfp\in Z_l}\mfp$. We claim that $\mfq$ is meet-prime. To see this, suppose that we are given $m,n\in L$ with $m\wedge n\leq \mfq$ i.e.\ $m\wedge n \leq \mfp$ for every $\mfp\in Z_l$, i.e. $Z_l\subseteq Z_{m\wedge n} = Z_m \cup Z_n$. By irreducibility of $Z_l$ we have, without loss of generality, that $Z_l\subseteq Z_m$, i.e.\ $m\leq \mfp$ for all $\mfp\in Z_l$ and so $m\leq \mfq$. Hence $\mfq$ is meet-prime as claimed and we have $Z_l = Z_{\mfq}$, so $Z_l$ has unique generic point $\mfq$. 
\end{proof}

So we have two ways of constructing a spatial frame from a complete lattice: on one hand $\lefter$ leaves lattices which are already spatial frames unchanged but we need to be careful to take maps that also preserve finite meets, while $\lefta$ is more brutal but we can work with arbitrary join preserving maps.

However, $\lefter$ can fail to be kind to nondistributive lattices.

\begin{ex}\label{ex:diamond}
Consider the diamond lattice $D$ with Hasse diagram
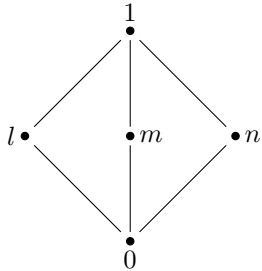
\begin{displaymath}
	\begin{tikzpicture}[scale=0.7]
    
		\node (v0) at (0,0) {};
    \node (va) at (-2,2) {};
    \node (vb) at (0,2) {};
		\node (vc) at (2,2) {};
		\node (v1) at (0,4) {};
    
    \draw[fill] (v0)  circle (2pt) node [below] {0};
		\draw[fill] (va)  circle (2pt) node [left] {$l$};
		\draw[fill] (vb)  circle (2pt) node [right] {$m$};
		\draw[fill] (vc)  circle (2pt) node [right] {$n$};
		\draw[fill] (v1)  circle (2pt) node [above] {$1$};

		\path[-] (v0) edge  node [above] {} (va);
		\path[-] (v0) edge  node [above] {} (vb);
		\path[-] (v0) edge  node [above] {} (vc);
		\path[-] (va) edge  node [above] {} (v1);
		\path[-] (vb) edge  node [above] {} (v1);
		\path[-] (vc) edge  node [above] {} (v1);

  \end{tikzpicture}
\end{displaymath}
This is a complete nondistributive lattice. We have $\CLat(D,\Sierp)=\varnothing$ since preserving meets and joins implies that any point $p$ would, for each pair of elements $i,j\in \{l,m,n\}$, have to send one of them to $1$ and the other to $0$. This cannot be done compatibly for all three of $l,m$, and $n$. Thus $\lefter D$ is $\mcO(\varnothing) = \{\varnothing\} = \mathbf{1}$ which is the terminal frame. 
\end{ex}

\begin{lem}\label{lem:serialkiller}
Let $L\in \CLat$ be a complete lattice. If there is an embedding of complete lattices $D\to L$ then $\lefter L \cong \mathbf{1}$.
\end{lem}
\begin{proof}
If $p\colon L\to \Sierp$ were a point of $L$ then it would restrict to a point of $D$. This is impossible, and so we must have $\pt(L) = \varnothing$ and the conclusion follows. 
\end{proof}

\begin{rem}
The lemma requires that the embedding takes place in $\CLat$ and so it must preserve the top and bottom elements. As the following example shows a lattice can be nondistributive, due to containing the diamond as a non-bounded sublattice (i.e.\ the diamond is a subposet), and be non-trivial under $\lefter$.
\end{rem}

\begin{ex}
Consider the complete nondistributive lattice $L$ with Hasse diagram
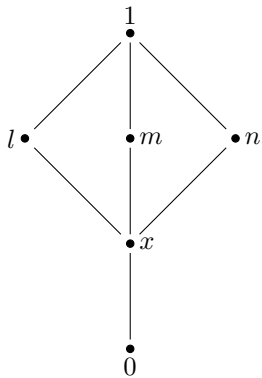
\begin{displaymath}
	\begin{tikzpicture}[scale=0.7]
    
		\node (v0) at (0,-2) {};
		\node (vx) at (0,0) {};
    \node (va) at (-2,2) {};
    \node (vb) at (0,2) {};
		\node (vc) at (2,2) {};
		\node (v1) at (0,4) {};
    
		\draw[fill] (vx)  circle (2pt) node [right] {$x$};
    \draw[fill] (v0)  circle (2pt) node [below] {0};
		\draw[fill] (va)  circle (2pt) node [left] {$l$};
		\draw[fill] (vb)  circle (2pt) node [right] {$m$};
		\draw[fill] (vc)  circle (2pt) node [right] {$n$};
		\draw[fill] (v1)  circle (2pt) node [above] {$1$};

    \path[-] (v0) edge  node [above] {} (vx);
		\path[-] (vx) edge  node [above] {} (va);
		\path[-] (vx) edge  node [above] {} (vb);
		\path[-] (vx) edge  node [above] {} (vc);
		\path[-] (va) edge  node [above] {} (v1);
		\path[-] (vb) edge  node [above] {} (v1);
		\path[-] (vc) edge  node [above] {} (v1);

  \end{tikzpicture}
\end{displaymath}
This lattice has a point $p\colon L\to \Sierp$ defined by $p(0) = 0$ and sending the remaining elements to $1$. In fact this is the unique point and so $\Spec L = \ast$, the terminal space, and $\lefter L = \Sierp$.
\end{ex}

The other minimal nondistributive lattice, the pentagon lattice, is not killed by $\lefter$.

\begin{ex}\label{ex:pentagon}
Let $P$ be the pentagon lattice with Hasse diagram
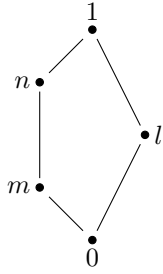
\begin{displaymath}
	\begin{tikzpicture}[scale=0.7]
    
		\node (v0) at (0,0) {};
    \node (va) at (-1,1) {};
    \node (vb) at (1,2) {};
		\node (vc) at (-1,3) {};
		\node (v1) at (0,4) {};
    
    \draw[fill] (v0)  circle (2pt) node [below] {0};
		\draw[fill] (va)  circle (2pt) node [left] {$m$};
		\draw[fill] (vb)  circle (2pt) node [right] {$l$};
		\draw[fill] (vc)  circle (2pt) node [left] {$n$};
		\draw[fill] (v1)  circle (2pt) node [above] {$1$};

		\path[-] (v0) edge  node [above] {} (va);
		\path[-] (v0) edge  node [above] {} (vb);
		\path[-] (va) edge  node [above] {} (vc);
		\path[-] (vb) edge  node [above] {} (v1);
		\path[-] (vc) edge  node [above] {} (v1);

  \end{tikzpicture}
\end{displaymath}
The lattice $P$ has precisely two points $p_1$ and $p_2$ given by
\begin{displaymath}
\begin{array}{cc}
m,n & \mapsto 0 \\
l & \mapsto 1
\end{array}
\quad \text{ and } \quad
\begin{array}{cc}
m,n &\mapsto 1 \\
l & \mapsto 0
\end{array}
\end{displaymath}
respectively, with the discrete topology. Thus $\lefter P$ is the lattice 
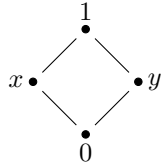
\begin{displaymath}
	\begin{tikzpicture}[scale=0.7]
    
		\node (v0) at (0,0) {};
    \node (va) at (-1,1) {};
    \node (vb) at (1,1) {};
		\node (v1) at (0,2) {};
    
    \draw[fill] (v0)  circle (2pt) node [below] {0};
		\draw[fill] (va)  circle (2pt) node [left] {$x$};
		\draw[fill] (vb)  circle (2pt) node [right] {$y$};
		\draw[fill] (v1)  circle (2pt) node [above] {$1$};

		\path[-] (v0) edge  node [above] {} (va);
		\path[-] (v0) edge  node [above] {} (vb);
		\path[-] (va) edge  node [above] {} (v1);
		\path[-] (vb) edge  node [above] {} (v1);

  \end{tikzpicture}
\end{displaymath}
where $x=\{p_2\}$ and $y = \{p_1\}$, and we see $m$ and $n$ have been collapsed to $x$.
\end{ex}



\section{Preliminaries on triangulated categories}\label{sec:t}
\setcounter{subsection}{1}

In this section, which serves mostly to fix notation and ideas, we introduce the lattices that will be our main focus.

We denote by $\Tcat$ the category of essentially small triangulated categories and exact functors between them. Throughout all triangulated categories will be essentially small unless it is explicitly mentioned otherwise.

Let $\sfK$ be a triangulated category.

\begin{defn}\label{def:thick}
A full subcategory $\sfM\subset \sfK$ is \emph{thick} if it is closed under:
\begin{itemize}
\item isomorphisms;
\item all suspensions;
\item cones;
\item direct summands,
\end{itemize} 
i.e.\ it is a triangulated subcategory closed under summands.

Given a collection of objects $C \subseteq \sfK$ we denote by $\thick(C)$ the smallest thick subcategory containing $C$ (which exists by Lemma~\ref{lem:thicklattice}). We will instead use $\thick_\sfK(C)$ if we wish to emphasise the ambient triangulated category. 
\end{defn}

We denote by $\Thick(\sfK)$ the collection of thick subcategories of $\sfK$. This is naturally a poset when ordered by inclusion.

\begin{lem}\label{lem:thicklattice}
The poset $\Thick(\sfK)$ is a complete lattice with meets given by intersections. The join of a family $\{\sfM_\lambda\mid \lambda \in \Lambda\}$ of thick subcategories is given by
\begin{displaymath}
\bigvee_{\lambda \in \Lambda} \sfM_\lambda = \thick(\bigcup_{\lambda \in \Lambda} \sfM_\lambda).
\end{displaymath}
\end{lem}
\begin{proof}
It is evident from the closure conditions defining a thick subcategory that any intersection of thick subcategories is again thick. Since a thick subcategory is determined by the objects it contains it follows immediately that this must give the meet.

That the join is as specified can be checked directly or deduced from the formula for the join in terms of the meet.
\end{proof}

\begin{rem}
The lattice $\Thick(\sfK)$ is in many examples not even distributive, let alone a frame; we explore this in depth in Section~\ref{sec:shithappens}.
\end{rem}

\begin{lem}\label{lem:commute}
Suppose $F\colon \sfK \to \sfL$ is an exact functor. For any collection of objects $C\subseteq \sfK$ there is an equality
\begin{displaymath}
\thick_{\sfL}(F\thick_{\sfK}(C)) = \thick_{\sfL}(FC).
\end{displaymath} 
\end{lem}
\begin{proof}
It is clear that $\thick_{\sfL}(FC) \subseteq \thick_{\sfL}(F\thick_{\sfK}(C))$. For the other direction, consider the subcategory
\begin{displaymath}
\sfM = \{k\in \sfK \mid Fk \in \thick_{\sfL}(FC)\}
\end{displaymath}
of $\sfK$. By (the analogue for thick subcategories of) \cite{StevensonActions}*{Lemma~3.8} the subcategory $\sfM$ is thick. It contains $C$ by construction and hence contains $\thick_{\sfK}(C)$. Thus $F\thick_{\sfK}(C) \subseteq \thick_{\sfL}(FC)$ from which the remaining containment follows immediately.
\end{proof}

\begin{lem}\label{lem:thickfunctor}
Suppose $F\colon \sfK \to \sfL$ is an exact functor. Then the assignment
\begin{displaymath}
T(F)\colon \Thick(\sfK) \to \Thick(\sfL) \quad \quad \sfM \mapsto \thick(F\sfM) =: T(F)\sfM
\end{displaymath}
is a map in $\CjSLat$, i.e.\ it preserves the order and arbitrary joins.
\end{lem}
\begin{proof}
If $\sfM \subseteq \sfN \subseteq \sfK$ are thick subcategories then $F\sfM \subseteq F\sfN$ and so clearly $T(F)$ is order preserving. Given a family $\{\sfM_\lambda\mid \lambda \in \Lambda\}$ of thick subcategories of $\sfK$ we have
\begin{align*}
T(F) \bigvee_{\lambda \in \Lambda} \sfM_\lambda &= \thick_\sfL\left(F \thick_{\sfK}\left(\bigcup_{\lambda \in \Lambda} \sfM_\lambda\right)\right) \\
&= \thick_\sfL\left(F \bigcup_{\lambda \in \Lambda} \sfM_\lambda\right) \\
&= \thick_\sfL\left(\bigcup_{\lambda \in \Lambda} F\sfM_\lambda\right) \\
&= \thick_\sfL\left(\bigcup_{\lambda \in \Lambda} \thick_{\sfL}\left( F\sfM_\lambda\right)\right) \\
&= \bigvee_{\lambda \in \Lambda} T(F)\sfM_\lambda,
\end{align*}
where we have used Lemma~\ref{lem:commute} for the second and fourth equalities.
\end{proof}

\begin{defn}
We define a functor $T\colon \Tcat \to \CjSLat$, by setting $T(\sfK) = \Thick(\sfK)$ and using Lemma~\ref{lem:thickfunctor} to define the action on maps.
\end{defn}

One cannot, in general, improve this to a factorization through $\CLat$. If one wants the assignment $\sfK \mapsto \Thick(\sfK)$ to be functorial with respect to all exact functors then the price one pays is that one has to ignore meets. This is fairly typical rather than pathological behaviour, which one should expect in examples where there is a `combinatorial' component to the classification problem for thick subcategories.

\begin{ex}\label{ex:letsnotmeet}
Consider the linearly oriented Dynkin quiver $A_2$
\begin{displaymath}
	\begin{tikzpicture}[shorten >= 2pt, shorten <= 2pt]
    
    \node (v1) at (1.5,0) {};
    \node (v2) at (3,0) {};
    
    \draw[fill] (v1)  circle (2pt) node [above] {\footnotesize 1};
    \draw[fill] (v2)  circle (2pt) node [above] {\footnotesize 2};

		\path[->] (v1) edge  node [above] {} (v2);

  \end{tikzpicture}
\end{displaymath}
and the quiver endomorphism $f$ determined by $f(1) = 2 = f(2)$. Let $F$ denote the corresponding derived base change functor $\sfD^\mathrm{b}(kA_2) \to \sfD^\mathrm{b}(kA_2)$. Let $P_1$ and $P_2$ denote the indecomposable projective modules. We have, on one hand,
\begin{align*}
T(F)\thick(P_1) \wedge T(F) \thick(P_2) &= \thick(FP_1) \wedge \thick (FP_2) \\
&= \thick(P_2) \wedge \thick(P_2) \\
&= \thick(P_2).
\end{align*}
But, on the other hand,
\begin{align*}
T(F)(\thick(P_1) \wedge \thick(P_2)) = T(F)(0) = 0.
\end{align*}
\end{ex}

\begin{rem}
It is also natural (perhaps seemingly more so given one's training) to consider, given $F\colon \sfK \to \sfL$, the poset map $F^{-1}\colon \Thick(\sfL) \to \Thick(\sfK)$ which preserves arbitrary meets (but not necessarily joins). Ordering thick subcategories by reverse containment one can thus define a functor $\Tcat^\op \to \CjSLat$.
However, the choice is in some sense cosmetic: the map $F^{-1}$ is right adjoint to $T(F)$ and so the functor $T$ and the functor $\Tcat^\op \to \CjSLat$, defined above, just differ by the standard duality on $\CjSLat$ (see \cite{sketches}*{Remark~1.1.7}). We regard the choice we make here, to prioritise $T$ when not inconvenient, as the correct one in the sense that it agrees with the choice implicitly made in algebraic and tt-geometry.

This is a little counterintuitive, so let us sketch the reasoning in the case of geometry. Let $f\colon R\to S$ be a map of commutative rings. Let $\Zar(R)$ denote the Zariski frame of radical ideals. There is a corresponding frame map $f\colon \Zar(R) \to \Zar(S)$ which, analogously to the above, sends an ideal $I$ to $\sqrt{(fI)}$. This gives rise to a continuous map $F\colon \pt(\Zar(S)) \to \pt(\Zar(R))$ by precomposition.

A point of $\Zar(S)$ is equivalent to a principal prime ideal in the lattice-theoretic sense, which is given by $I_\mfp = \{J\in \Zar(S) \mid J\subseteq \mfp\}$ for a usual prime ideal $\mfp$ of $S$. The corresponding point  of $\Zar(R)$, which is the composite $\Zar(R) \to \Zar(S) \to \Sierp$, is then given precisely by the meet-prime 
\begin{displaymath}
\bigvee \{K\in \Zar(R)\mid fK \leq \mfp\} = \bigvee \{K\in \Zar(R) \mid K\subseteq f^{-1}\mfp\} = f^{-1}\mfp.
\end{displaymath}
Thus, as claimed, the counterintuitive map on Zariski frames give rise to our beloved preimages on primes.
\end{rem}

Now let us give a very brief recapitulation of the examples motivating the constructions of this article.

\begin{defn}
By a \emph{tt-category} we mean a tensor triangulated category, i.e.\ a triangulated category equipped with a symmetric monoidal structure which is exact in each variable. 
\end{defn}

There are by now many surveys on tt-categories and their corresponding tt-geometry, for instance \cite{BaICM} and \cite{Stevensontour} and the references therein, and one can see the work of Kock and Pitsch \cite{KockPitsch} for a more lattice-theoretic point of view; given this we omit the majority of the details here.

They key facts are that the lattice of radical thick $\otimes$-ideals forms a coherent frame (by \cite{BKS}) and that under the functor $\pt$ of Stone duality this gives the Hochster dual of the Balmer spectrum.

The following should be well known (e.g.\ the statement can be found in \cite{HPS}), but we provide a proof for completeness.

\begin{lem}\label{lem:rigid}
Let $\sfT$ be a closed tt-category. Then the full subcategory of rigid objects $\sfT^\mathrm{rig}$ is thick. In particular, we have $\thick(\unit)\subseteq \sfT^\mathrm{rig}$.
\end{lem}
\begin{proof}
By definition $x\in \sfT$ is rigid if for all $y\in \sfT$ the natural map
\begin{displaymath}
\xymatrix{
x^\vee \otimes y \ar[r]^-{\alpha_{x,y}} & \hom(x,y)
}
\end{displaymath}
is an isomorphism. Consider then, for fixed $y\in \sfT$, the subcategory
\begin{displaymath}
\sfR_y = \{x\in \sfT \mid \alpha_{x,y} \text{ is an isomorphism.}\}.
\end{displaymath}
Since $\alpha_{x,y}$ is natural in $x$ and compatible with sums and suspension, the source functor is exact, and the target functor sends triangles to pre-triangles, we see that $\sfR_y$ is thick\textemdash{}this is a twist on the standard argument, exploiting the fact that, by \cite{NeeCat}*{Proposition~1.1.20}, one already has the 2-of-3 property for isomorphisms for maps of pre-triangles. It is then enough to note that
\begin{displaymath}
\sfT^\mathrm{rig} = \bigcap_{y\in \sfT} \sfR_y
\end{displaymath}
since an intersection of thick subcategories is thick.
\end{proof}

We close by discussing a class of examples that will be of use in the sequel. Consider a family of triangulated categories $\{\sfK_\lambda\mid \lambda\in \Lambda\}$ indexed by a set $\Lambda$. Then the coproduct and product, as additive categories, exist and inherit canonical pointwise triangulated structures. 

\begin{lem}\label{lem:product}
Let $\{\sfK_\lambda\mid \lambda\in \Lambda\}$ be a family of triangulated categories with coproduct $\sfK$. Then the natural map
\begin{displaymath}
\Thick(\sfK) \stackrel{\sim}{\to} \prod_{\lambda\in \Lambda} \Thick(\sfK_\lambda),
\end{displaymath}
is an isomorphism of complete lattices, i.e.\ it is a poset isomorphism preserving finite meets and arbitrary joins.
\end{lem}
\begin{proof}
The canonical inclusions $\sfK_\lambda \to \sfK$ give, via taking preimages, complete meet-semilattice maps $\Thick(\sfK) \to \Thick(\sfK_\lambda)$ and so there is a canonical map of meet-semilattices as indicated in the statement, which we call $\phi$.

Let us prove that $\phi$ is bijective. If $X = \{\sfM_\lambda\mid\lambda\in \Lambda\}$ is an element of $\prod_{\lambda\in \Lambda} \Thick(\sfK_\lambda)$, then there is a fully faithful exact functor
\begin{displaymath}
\sfM = \bigoplus_\lambda \sfM_\lambda \to \sfK,
\end{displaymath}
sending an object on the left to the direct sum (which is finite) of its components, and we set $\psi(X) = \sfM$ which defines a map
\[
\psi \colon \prod_{\lambda\in \Lambda} \Thick(\sfK_\lambda) \to \Thick(\sfK)
\]
It is clear that $\phi\psi(X) = X$. On the other hand, if $\sfN$ is a thick subcategory of $\sfK$ then, by summand closure, we certainly have $\psi\phi(\sfN) \subseteq \sfN$. This is an equality since every object of $\sfK$ is a finite sum of objects of the $\sfK_\lambda$. 


Finally, we need to check that $\phi$ also preserves all joins. There are several ways to deduce this; the most direct is to notice it follows, more or less immediately, from the componentwise nature of the triangulated structure. 
\end{proof}

\begin{rem}
Understanding the product $\prod_{\lambda} \sfK_\lambda$ is another story. This contains the direct sum as a thick subcategory, and hence a copy of $\prod_{\lambda} \Thick(\sfK_\lambda)$ as an (unbounded) sublattice, but the lattice $\Thick(\prod_{\lambda} \sfK_\lambda)$ is strictly larger (assuming infinitely many $\sfK_\lambda$ are non-zero) and seems significantly more complicated.
\end{rem}



\section{Most possibilities occur}\label{sec:shithappens}

In this section we present examples illustrating that various properties may or may not hold for the lattice of thick subcategories of a triangulated category $\sfK$. We start with the best possible, from the point of view of pointless topology, and work our way down the hierarchy. 


\subsection{Frames}

It is well known, by the initiated, that the lattice of thick radical $\otimes$-ideals in a tt-category forms a coherent frame (see e.g.\ \cite{BKS}). It follows that, in special cases, the entire lattice of thick subcategories forms a coherent frame.

\begin{lem}\label{lem:unit}
If $\sfT$ is a closed tt-category such that $\sfT = \thick(\unit)$ then every thick subcategory is a $\otimes$-ideal and $\sfT$ is rigid. Hence $\Thick(\sfT)$ is a coherent frame.
\end{lem}
\begin{proof}
The first part is standard, and the second is an application of Lemma~\ref{lem:rigid}. The final statement then follows by the identification of $\Thick(\sfT)$ with the lattice of radical $\otimes$-ideals, which is isomorphic to the frame of Thomason subsets of $\Spc \sfT$ by \cite{BaSpec}*{Theorem~4.10}. The latter is always a coherent frame. 
\end{proof}

\begin{ex}
For instance, we could take $R$ to be a commutative ring and $\sfK = \sfD^\mathrm{perf}(R)$, which is generated by the tensor unit $R$. Then $\Thick(\sfD^\mathrm{perf}(R))$ is a coherent frame, and in fact is dual to the Zariski frame of radical ideals of $R$ (see \cite{KockPitsch} for further discussion).
\end{ex}

Upon seeing this one might hope that, for a closed tt-category $\sfT$, the lattice $\Thick(\sfT)$ is a frame precisely when $\unit$ generates. As the following example shows this is not at all the case. We can use Lemma~\ref{lem:product} to produce non-trivial examples of tt-categories $\sfT$ such that $\Thick(\sfT)$ is a frame and $\Thick^\otimes(\sfT)\subsetneq \Thick(\sfT)$, so in particular they are not generated by $\unit$.

\begin{ex}
Let $R$ be a commutative noetherian ring and let $G$ be a finite abelian group (abelian is only necessary so that the monoidal structure we obtain is symmetric). We construct the tt-analogue of the group algebra: consider the `group tt-category' $\sfD^\mathrm{perf}(R)G$
\begin{displaymath}
\sfD^\mathrm{perf}(R)G = \bigoplus_{g\in G} \sfD^\mathrm{perf}(R)_g,
\end{displaymath}
where the $g$ is just an index to let us keep track of components. This is triangulated, with the componentwise triangulated structure, and we define the tensor product by
\begin{displaymath}
X_g \otimes Y_h = (X\otimes_R Y)_{gh}
\end{displaymath}
and additivity.

One could, equivalently, enhance the situation and view this as the homotopy category of the functor category $[G,\sfD^\mathrm{perf}(R)]$ with the Day convolution monoidal structure, or as the compact objects of the derived category of $G$-graded $R$-modules, where $R$ is equipped with the trivial grading, with the derived tensor product of graded modules. 

By Lemma~\ref{lem:product} we have $\Thick(\sfD^\mathrm{perf}(R)G) \cong \prod_{g\in G} \Thick(\sfD^\mathrm{perf}(R))$, i.e. a product of $\vert G\vert$ copies of the frame of Thomason subsets of $\Spec R$.

On the other hand $\Thick^\otimes(\sfD^\mathrm{perf}(R)G)$ is a single copy of the frame of Thomason subsets, as follows from either direct computation, \cite{DS13}*{Theorem~5.1} if viewing this as the derived category of $G$-graded modules, or \cite{aoki}*{Theorem~I} from the functor category point of view.
\end{ex}


\subsection{Distributivity}

It is natural to ask if one can find examples where $\Thick(\sfK)$ is only a frame, i.e.\ not coherent, or only distributive. It turns out that actually one gets a lot for free: we will show that distributivity of $\Thick(\sfK)$ implies that $\Thick(\sfK)$ is a frame. We start with a standard fact.

\begin{lem}\label{lem:finite}
Let $\{\sfM_\lambda\mid \lambda \in \Lambda\}$ be a family of thick subcategories and suppose that
\begin{displaymath}
k\in \bigvee_{\lambda \in \Lambda} \sfM_\lambda.
\end{displaymath}
Then there exist $\lambda_1,\ldots, \lambda_n$ such that $k\in \sfM_{\lambda_1} \vee \cdots \vee \sfM_{\lambda_n}$.
\end{lem}
\begin{proof}
By definition
\begin{displaymath}
k\in \bigvee_{\lambda \in \Lambda} \sfM_\lambda = \thick\left( \: \bigcup_{\lambda \in \Lambda} \sfM_\lambda\right) \!.
\end{displaymath}
We can construct the latter inductively, in the usual fashion, by starting with the $\sfM_\lambda$ and taking iterated cones. Each cone involves only a finite sum of objects from the $\sfM_\lambda$, and our fixed object $k$ occurs after finitely many steps, and so the statement follows.
\end{proof}

\begin{prop}\label{prop:dist}
Let $\sfK$ be an essentially small triangulated category such that $\Thick(\sfK)$ is distributive. Then $\Thick(\sfK)$ is a frame.
\end{prop}
\begin{proof}
Suppose we are given a thick subcategory $\sfL$ and a family $\{\sfM_\lambda\mid \lambda \in \Lambda\}$ of thick subcategories. It is always true that
\begin{displaymath}
\sfL \cap \bigvee_{\lambda \in \Lambda} \sfM_\lambda \supseteq \bigvee_{\lambda \in \Lambda} (\sfL \cap \sfM_\lambda),
\end{displaymath}
so it is sufficient to check the reverse inclusion. Let $k$ be an object of the category on the left. In particular, $k$ lies in $\bigvee_{\lambda \in \Lambda} \sfM_\lambda$ and so by Lemma~\ref{lem:finite} there exist $\lambda_1,\ldots, \lambda_n$ such that $k$ lies in $\bigvee_{i=1}^n \sfM_{\lambda_i}$. Thus
\begin{displaymath}
k \in \sfL \cap \left(\bigvee_{i=1}^n \sfM_{\lambda_i}\right) = \bigvee_{i=1}^n(\sfL \cap \sfM_{\lambda_i})
\end{displaymath}
by distributivity of $\Thick(\sfK)$. Of course we have
\begin{displaymath}
\bigvee_{i=1}^n(\sfL \cap \sfM_{\lambda_i}) \subseteq \bigvee_{\lambda \in \Lambda} (\sfL \cap \sfM_\lambda)
\end{displaymath}
and so $k$ lies in $\sfL \cap \bigvee_{\lambda \in \Lambda} \sfM_\lambda$ proving the desired equality. 
\end{proof}

\begin{rem}
This is really a special case of a more general statement about lattices (standard lore to a different crowd). If $L$ is an algebraic (aka compactly generated) lattice, then if it is distributive it is automatically a frame (cf.\ Lemma~\ref{lem:compact} for the fact that $\Thick(\sfK)$ is algebraic). In fact, one can even deduce that such a frame is spatial; we give a proof imminently.
\end{rem}

\subsection{Spatiality}

We have seen that distributivity implies infinite distributivity for lattices of thick subcategories. We next discuss the property of being spatial; it turns out that this too is for free.

\begin{thm}\label{thm:spatial}
Suppose that $\Thick(\sfK)$ is distributive. If $s\in \sfK$ and $\sfM$ is a thick subcategory of $\sfK$ with $s\notin \sfM$ then there exists a meet-prime thick subcategory $\sfP$ such that $\sfM\subseteq \sfP$ and $s\notin \sfP$.
\end{thm}
\begin{proof}
Let us, with a view to using Zorn's lemma, consider the set
\[
\mcF = \{\sfL \in \Thick(\sfK) \mid s\notin \sfL \text{ and } \sfM\subseteq \sfL\}.
\]
Clearly $\sfM \in \mcF$ and so $\mcF$ is not empty. We will show that every chain in $\mcF$ has an upper bound. Suppose then that $\{\sfL_i \mid i\in I\}$ is a chain in $\mcF$. Since being thick is a family of conditions on finite sets of objects, the category $\sfL = \cup_i \sfL_i$ is thick. As $s\notin \sfL_i$ for any $i\in I$ we have $s\notin \sfL$ and it is clear that $\sfM\subseteq \sfL$. Thus $\sfL\in \mcF$ and is an upper bound for the given chain.

By Zorn's lemma $\mcF$ contains a maximal element $\sfP$, which we shall show is meet-prime. Suppose then that $\sfA$ and $\sfB$ are thick subcategories with $\sfA\cap \sfB \subseteq \sfP$. We have $\sfP \subseteq \sfP \vee \sfA$ and so if $\sfA \nsubseteq \sfP$, by maximality of $\sfP$ in $\mcF$, we must have $s\in \sfP \vee \sfA$ to avoid a contradiction, and similarly for $\sfB$. With this in mind, consider the equation
\[
\sfP = \sfP \vee (\sfA \cap \sfB) = (\sfP \vee \sfA) \cap (\sfP \vee \sfB),
\]
where the first equality follows from $\sfA\cap \sfB \subseteq \sfP$ and the second from distributivity. If neither $\sfA$ nor $\sfB$ are contained in $\sfP$ then, by what we saw above, we have $s\in \sfP \vee \sfA$ and $s \in \sfP \vee \sfB$ and so $s\in \sfP$. But this is absurd, and so at least one of $\sfA$ or $\sfB$ is contained in $\sfP$. This shows that $\sfP$ is meet-prime as claimed.
\end{proof}

\begin{cor}\label{cor:spatial}
If $\Thick(\sfK)$ is distributive then it is a spatial frame.
\end{cor}
\begin{proof}
Suppose that $\sfM, \sfN \in \Thick(\sfK)$ are such that $\sfN \nsubseteq \sfM$. Then there is an $s\in \sfN$ with $s\notin \sfM$. By the Theorem we can find a meet-prime $\sfP$ with $\sfM\subseteq \sfP$ and $s\notin \sfP$. It follows that $\sfN\nsubseteq \sfP$ and the point $p$ of $\Thick(\sfK)$ corresponding to $\sfP$ satisfies $p(\sfM) = 0$ and $p(\sfN) = 1$ as required.
\end{proof}

The following statement is immediate from this last corollary, but striking enough to bear repeating.

\begin{cor}\label{cor:sober}
If $\Thick(\sfK)$ is distributive then there is a sober topological space, namely $X = \Spec \Thick(\sfK)$, such that $\Thick(\sfK)$ is isomorphic to the lattice of open subsets of $X$. 
\end{cor}

\subsection{Coherence}

We have shown that as soon as $\Thick(\sfK)$ is distributive it is necessarily a spatial frame. It is then natural to ask about coherence. This turns out to be a more delicate issue and can fail, at least in principle, in a couple of ways.

The first matter is to identify the compact thick subcategories, which is independent of any distributivity hypothesis. Such subcategories have appeared in various spots in the literature, usually under the moniker `finitely generated'.

\begin{lem}\label{lem:compact}
A thick subcategory $\sfL\in \Thick(\sfK)$ is compact if and only if there exists a $k\in \sfK$ such that $\sfL = \thick(k)$.
\end{lem}
\begin{proof}
Suppose first that $\sfL$ is compact. We have
\begin{displaymath}
\sfL = \bigvee_{l\in \sfL} \thick(l).
\end{displaymath}
By compactness there exist $l_1,\ldots, l_n$ in $\sfL$ such that 
\begin{displaymath}
\sfL = \bigvee_{i=1}^n \thick(l_i)
\end{displaymath}
and so we can take $k= l_1\oplus \cdots \oplus l_n$.

On the other hand, suppose that $\sfL = \thick(k)$ and we have a family of thick subcategories $\{\sfM_\lambda\mid \lambda \in \Lambda\}$ with
\begin{displaymath}
\sfL \subseteq \bigvee_{\lambda \in \Lambda} \sfM_\lambda.
\end{displaymath}
By Lemma~\ref{lem:finite} we can find $\lambda_1,\ldots, \lambda_n$ such that $k$ is in the thick subcategory generated by the $\sfM_{\lambda_i}$ for $1\leq i \leq n$. This subcategory is thick and so contains $\sfL = \thick(k)$, i.e.\ we have
\begin{displaymath}
\sfL \subseteq \bigvee_{i=1}^n \sfM_{\lambda_i}. 
\end{displaymath}
\end{proof}

\begin{rem}
We will sometimes adopt the common terminology mentioned above and call a thick subcategory $\sfL$ of the form $\sfL = \thick(k)$ finitely generated (due to the presence of finite direct sums being finitely generated is the same as being principal).
\end{rem}

\begin{rem}
We note that, since every thick subcategory is the union of the thick subcategories generated by each of its objects, $\Thick(\sfK)$ is always an algebraic lattice, i.e.\ every element is a join of compact ones.
\end{rem}

It follows from the lemma that any finite join of compact objects is again compact. However, coherence requires that the compact objects form a bounded sublattice; the following example illustrates the most naive obstruction.

\begin{ex}
Fix a field $k$ and consider $\sfK = \sfD_\mathrm{tors}^\b(k[x])$ the derived category of bounded complexes of $k[x]$-modules with finitely generated torsion cohomology. This is a thick subcategory of $\sfD^\b(k[x])$ and so $L = \Thick(\sfK)$ is distributive, and hence a frame, by virtue of being a sublattice of the coherent frame $\Thick(\sfD^\b(k[x]))$.

However, $L$ is not coherent. Indeed $\sfD_\mathrm{tors}^\b(k[x])$ is not finitely generated and hence not compact: there is a direct sum decomposition
\[
\sfD_\mathrm{tors}^\b(k[x]) = \bigoplus_{\alpha} \thick(k(\alpha))
\]
where $\alpha$ runs over the closed points of $\AA^1_k$. The corresponding space, $\Spec L$, is in bijection with the closed points of $\AA^1_k$ and has the discrete topology. In particular, it is not quasi-compact. 
\end{ex}

There is also another potential subtlety: even if $\Thick(\sfK)$ is distributive and $\sfK$ itself is finitely generated it is not clear that an intersection of finitely generated thick subcategories remains finitely generated. We are not aware of an example where this goes awry, but suspect such an example exists.


\subsection{Modularity}

Let $L$ be a complete lattice. We recall that $L$ is \emph{modular} if for any triple $l,m,s\in L$ such that $l\leq m$ we have $l\vee (s\wedge m) = (l\vee s)\wedge m.$ This is a weakening of distributivity, and is satisfied in a number of contexts, for instance lattices of submodules are always modular. While the analogy between thick subcategories and lattices of ideals or submodules has been quite fruitful, it does not extend to include this observation.

\begin{ex}\label{ex:notmodular}
Consider $\sfK = \sfD^\b(\PP^1)$ over some fixed, but nameless, ground field. The lattice of thick subcategories of $\sfK$ is given, as a set, by
\begin{displaymath}
\{\text{specialization closed subsets of } \PP^1\} \coprod \ZZ,
\end{displaymath}
where the first component classifies tensor ideals, the second is given by the twisting sheaves, and the two pieces interact trivially, see \cite{KS19}*{Section~4.1} for further discussion. An embedding of the minimal non-modular lattice is given as follows
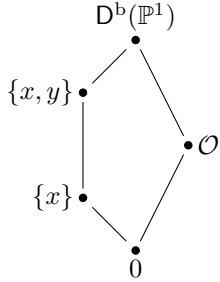
\begin{displaymath}
	\begin{tikzpicture}[scale=0.7]
    
		\node (v0) at (0,0) {};
    \node (va) at (-1,1) {};
    \node (vb) at (1,2) {};
		\node (vc) at (-1,3) {};
		\node (v1) at (0,4) {};
    
    \draw[fill] (v0)  circle (2pt) node [below] {0};
		\draw[fill] (va)  circle (2pt) node [left] {$\{x\}$};
		\draw[fill] (vb)  circle (2pt) node [right] {$\mcO$};
		\draw[fill] (vc)  circle (2pt) node [left] {$\{x,y\}$};
		\draw[fill] (v1)  circle (2pt) node [above] {$\sfD^\b(\PP^1)$};

		\path[-] (v0) edge  node [above] {} (va);
		\path[-] (v0) edge  node [above] {} (vb);
		\path[-] (va) edge  node [above] {} (vc);
		\path[-] (vb) edge  node [above] {} (v1);
		\path[-] (vc) edge  node [above] {} (v1);

  \end{tikzpicture}
\end{displaymath}
where $x,y\in \PP^1$ are distinct points, the subsets indicate the corresponding tensor ideals of objects supported on those points, and $\mcO$ is shorthand for the thick subcategory it generates, showing that this lattice is non-modular.
\end{ex}
In fact this lattice is not even \emph{semi-modular}. We say that $l$ covers $m$ if $l> m$ and $l$ is minimal with this property, i.e. $\{n \in L \mid m\leq n \leq l\} = \{m,l\}$. A lattice is semi-modular if for all $l,m\in L$ we have that $l$ covers $l\wedge m$ implies that $m$ is covered by $l\vee m$. Indeed, taking $l$ to be $\thick(\mcO)$ and $m$ to be the thick subcategory generated by two distinct points gives a violation of this requirement.

However, there are examples of lattices of thick subcategories which are not distributive but do satisfy the modular law.

\begin{ex}\label{ex:onlymodular1}
The lattice of thick subcategories of $\sfD^\b(kA_2)$, for a field $k$, is given by the diamond lattice (as depicted in Example~\ref{ex:diamond}). This is not distributive, but it is modular (for trivial reasons).
\end{ex}

\begin{rem}
We note that $\Thick(\sfD^\b(kA_n))$ is not even semi-modular for $n\geq 3$.
\end{rem}

Due to our dearth of understanding of $\Thick(\sfK)$ in non-trivial non-distributive examples it is difficult to provide further examples. However, modularity corresponds to a strong, but entirely natural, closure condition.

\begin{lem}\label{lem:modular}
Let $\sfK$ be a triangulated category. Then $\Thick(\sfK)$ is modular if and only if for any $\sfL,\sfM,\sfS \in \Thick(\sfK)$ such that $\sfL \subseteq \sfM$, then $X\in \thick(\sfL, \sfS)\cap \sfM$ implies that $X\in \thick(\sfL, \sfS\cap \sfM)$. That is to say, if $X\in \sfM$ is built from $\sfL$ and $\sfS$ then $X$ is built from objects in $\sfL$ and $\sfS\cap \sfM$. 
\end{lem}
\begin{proof}
Let $\sfL,\sfM$, and $\sfS$ be as in the statement. The modular inequality
\begin{displaymath}
\sfL \vee (\sfS \wedge \sfM) \leq (\sfL \vee \sfS)\wedge \sfM
\end{displaymath}
always holds. Indeed, this reads
\begin{displaymath}
\thick(\sfL,\sfS\cap \sfM) \subseteq \thick(\sfL,\sfS)\cap \sfM,
\end{displaymath}
which holds as $\sfL\subseteq \sfM$ so the left-hand side is contained in $\sfM$, and is evidently contained in $\thick(\sfL,\sfS)$, so is contained in their intersection.

Thus $\Thick(\sfK)$ is modular precisely if the right-hand side is contained in the left. This means that for $X\in \thick(\sfL,\sfS)\cap \sfM$ we must have $X\in \thick(\sfL,\sfS\cap \sfM)$.
\end{proof}

We conclude the section with a plentiful source of non-distributive (and, in general, most likely non-modular) examples. The necessary background for the statement and proof can be found in \cite{BondalReps}*{Section~2}.

\begin{lem}
Let $\sfK$ be a triangulated category and $k$ a field. If $\sfK$ is $k$-linear and Hom-finite with an exceptional pair $(E_1,E_2)$ such that $\sfK(E_1, \Sigma^iE_2)\neq 0$ for some $i$ then $\Thick(\sfK)$ is not distributive.
\end{lem}
\begin{proof}
Set $\sfE = \thick(E_1,E_2)$. Given the corresponding injection $\Thick(\sfE) \to \Thick(\sfK)$, which preserves binary meets and joins, it is enough to show that $\Thick(\sfE)$ is not distributive. Consider then the left mutation $F = L_{E_1}E_2$ of $E_2$ along $E_1$. This is again an exceptional object and, by the assumption that $\sfK(E_1, \Sigma^iE_2)\neq 0$, $F$ is not isomorphic to a shift of $E_1$ or $E_2$. The thick subcategory generated by an exceptional object $E$ is just $\add(\Sigma^i E \mid i\in \ZZ)$ and so we obtain a copy 
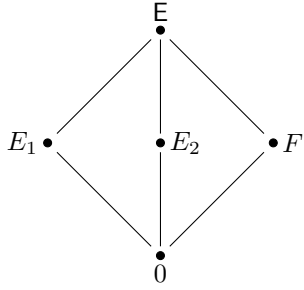
\begin{displaymath}
	\begin{tikzpicture}[scale=0.75]
    
		\node (v0) at (0,0) {};
    \node (va) at (-2,2) {};
    \node (vb) at (0,2) {};
		\node (vc) at (2,2) {};
		\node (v1) at (0,4) {};
    
    \draw[fill] (v0)  circle (2pt) node [below] {0};
		\draw[fill] (va)  circle (2pt) node [left] {$E_1$};
		\draw[fill] (vb)  circle (2pt) node [right] {$E_2$};
		\draw[fill] (vc)  circle (2pt) node [right] {$F$};
		\draw[fill] (v1)  circle (2pt) node [above] {$\sfE$};

		\path[-] (v0) edge  node [above] {} (va);
		\path[-] (v0) edge  node [above] {} (vb);
		\path[-] (v0) edge  node [above] {} (vc);
		\path[-] (va) edge  node [above] {} (v1);
		\path[-] (vb) edge  node [above] {} (v1);
		\path[-] (vc) edge  node [above] {} (v1);

  \end{tikzpicture}
\end{displaymath}

of the diamond lattice (where we have just used the labels for the exceptional objects to indicate the thick subcategories they generate). This witnesses the failure of $\Thick(\sfE)$ to be distributive.
\end{proof}


%



\section{The fully functorial theory}\label{sec:ff}

In this section we use the adjunction between $\SFrm$ and $\CjSLat$ to assign a sober space to any essentially small triangulated category. This construction is universal and gives some measure of the lattice of thick subcategories. As we shall see it is a rather imperfect measure, but it has the advantage of being functorial with respect to all exact functors. 

\subsection{The construction}

We have constructed in Section~\ref{sec:t} the functor $T\colon \Tcat \to \CjSLat$, which sends a triangulated category $\sfK$ to its lattice of thick subcategories $T(\sfK) = \Thick(\sfK)$. We can combine this with our lattice theoretic noodling to get somewhere.

\begin{defn}\label{def:fspcnt}
We define the \emph{fully functorial non-tensor spectrum}
\begin{displaymath}
\fSpcnt\colon \Tcat^\op \to \Sob
\end{displaymath}
to be the composite
\begin{displaymath}
\xymatrix{
\Tcat^\op \ar[r]^-{T^\op} & \CjSLat^\op \ar[r]^-{\lefta^\op} & \SFrm^\op \ar[r]^-{\sim} & \Sob
}
\end{displaymath}
where $\lefta$ is the adjoint of Proposition~\ref{prop:adjoint1}.
\end{defn}

\begin{rem}
Let us explain the notation $\fSpcnt$: the f denotes functorial (or free, since this construction is as free as is reasonable in contrast to the construction of Section~\ref{sec:pf}) and the nt is for `no tensor' (to be read as $\Spc$n't, as in `tensorn't triangular geometry', with the consequent abbreviation tnt-geometry).
\end{rem}

\begin{ex}\label{ex:fspcnta2}
Consider the Dynkin quiver $A_2$ as in Example~\ref{ex:letsnotmeet} and denote by $P_1$ and $P_2$ the projectives and by $S_2$ the non-projective simple. Then the lattice of thick subcategories of $\sfD^\mathrm{b}(kA_2)$ is 
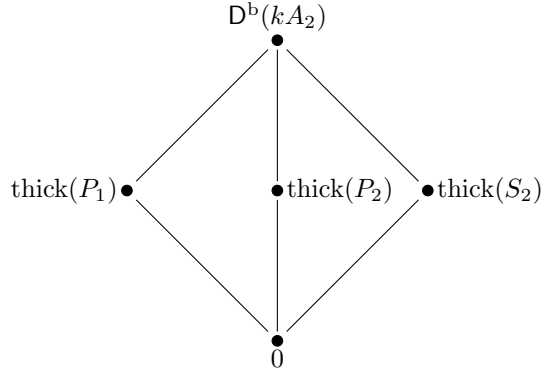
\begin{displaymath}
	\begin{tikzpicture}
    
		\node (v0) at (0,0) {};
    \node (va) at (-2,2) {};
    \node (vb) at (0,2) {};
		\node (vc) at (2,2) {};
		\node (v1) at (0,4) {};
    
    \draw[fill] (v0)  circle (2pt) node [below] {0};
		\draw[fill] (va)  circle (2pt) node [left] {$\thick(P_1)$};
		\draw[fill] (vb)  circle (2pt) node [right] {$\thick(P_2)$};
		\draw[fill] (vc)  circle (2pt) node [right] {$\thick(S_2)$};
		\draw[fill] (v1)  circle (2pt) node [above] {$\sfD^\mathrm{b}(kA_2)$};

		\path[-] (v0) edge  node [above] {} (va);
		\path[-] (v0) edge  node [above] {} (vb);
		\path[-] (v0) edge  node [above] {} (vc);
		\path[-] (va) edge  node [above] {} (v1);
		\path[-] (vb) edge  node [above] {} (v1);
		\path[-] (vc) edge  node [above] {} (v1);

  \end{tikzpicture}
\end{displaymath}
Note that $L = T(\sfD^\mathrm{b}(kA_2))$ is not distributive, and hence not a frame. We apply Construction~\ref{cons:left} to compute $\fSpcnt \sfD^\mathrm{b}(kA_2)$.

For brevity, let us denote by $0$ and $1$ the bottom and top elements and set
\begin{displaymath}
a= \thick(P_1), \; b = \thick(P_2), \text{ and } c = \thick(S_2).
\end{displaymath}
We need to compute $\spt(L) = \CjSLat(L, \Sierp)$. Suppose that $p\in \spt(L)$ satisfies $p(1)=1$. Then since $a\vee b = 1$ at least one of $a$ or $b$ must be sent to $1$. The roles of $a,b,c$ are interchangeable and so we see that if $p(1)=1$ then $p$ sends at most one of $a,b,c$ to $0$ and this determines $p$. This gives rise to $4$ points $\{p_a, p_b, p_c, p_\varnothing\}$ determined by sending $a$, $b$, and $c$ to $0$ in the first three instances, and sending all of $a,b,c$ to $1$ in the final instance. There is a fifth point $p_0$ sending all elements to $0$.

A subbase of opens is given by
\begin{displaymath}
U_a = \{p_b, p_c, p_\varnothing\}, U_b = \{p_a, p_c, p_\varnothing\}, U_c = \{p_a,p_b,p_\varnothing\}, \text{ and } U_1 = \{p_a,p_b,p_c, p_\varnothing\}.
\end{displaymath}

Thus we obtain the following space (where an edge indicates a specialization relation, i.e.\ that the top vertex of an edge is in the closure of the bottom one):
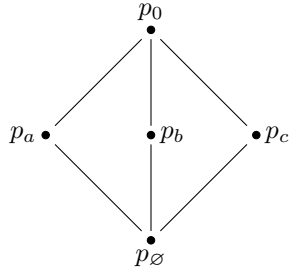
\begin{displaymath}
	\begin{tikzpicture}[scale=0.7]
    
		\node (v0) at (0,0) {};
    \node (va) at (-2,2) {};
    \node (vb) at (0,2) {};
		\node (vc) at (2,2) {};
		\node (v1) at (0,4) {};

    \draw[fill] (v0)  circle (2pt) node [below] {$p_\varnothing$};
		\draw[fill] (va)  circle (2pt) node [left] {$p_a$};
		\draw[fill] (vb)  circle (2pt) node [right] {$p_b$};
		\draw[fill] (vc)  circle (2pt) node [right] {$p_c$};
		\draw[fill] (v1)  circle (2pt) node [above] {$p_0$};

		\path[-] (v0) edge  node [above] {} (va);
		\path[-] (v0) edge  node [above] {} (vb);
		\path[-] (v0) edge  node [above] {} (vc);
		\path[-] (va) edge  node [above] {} (v1);
		\path[-] (vb) edge  node [above] {} (v1);
		\path[-] (vc) edge  node [above] {} (v1);

  \end{tikzpicture}
\end{displaymath}
with a unique generic point $p_\varnothing$ and unique closed point $p_0$.
\end{ex}

This looks rather a lot like the lattice we started with. We show in Proposition~\ref{prop:L} that this is not a coincidence.


\subsection{Failure to reconstruct the Balmer spectrum}\label{ssec:fail}

In this section we discuss the price for working with arbitrary exact functors: if $\sfK$ already has a spatial frame of thick subcategories, and hence has a space controlling the thick subcategories, then $\fSpcnt$ fails to recover this space. This is for the simple reason that the forgetful functor $i\colon \SFrm \to \CjSLat$ is only faithful and not full. Thus, for a spatial frame $F$, the counit
\begin{displaymath}
\varepsilon_F\colon \lefta i F \to F
\end{displaymath}
is an epimorphism, but cannot be an isomorphism in general.

\begin{ex}
We consider a pair of points, i.e.\ $\sfD^\mathrm{b}(k\times k) \cong \sfD^\mathrm{b}(k) \oplus \sfD^\mathrm{b}(k)$. The Balmer spectrum is a disjoint union of two points, and the corresponding frame (of opens or, as we label it, of radical $\otimes$-ideals) is
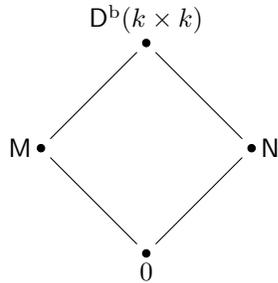
\begin{displaymath}
\begin{tikzpicture}[scale=0.7]
    
		\node (v0) at (0,0) {};
    \node (va) at (-2,2) {};
		\node (vc) at (2,2) {};
		\node (v1) at (0,4) {};
    
    \draw[fill] (v0)  circle (2pt) node [below] {0};
		\draw[fill] (va)  circle (2pt) node [left] {$\sfM$};
		\draw[fill] (vc)  circle (2pt) node [right] {$\sfN$};
		\draw[fill] (v1)  circle (2pt) node [above] {$\sfD^\mathrm{b}(k\times k)$};

		\path[-] (v0) edge  node [above] {} (va);
		\path[-] (v0) edge  node [above] {} (vc);
		\path[-] (va) edge  node [above] {} (v1);
		\path[-] (vc) edge  node [above] {} (v1);

  \end{tikzpicture}
\end{displaymath}
where $\sfM$ and $\sfN$ correspond to the two factors. The space $\fSpcnt\sfD^\mathrm{b}(k\times k)$, which is computed in a manner analogous to that used in Example~\ref{ex:fspcnta2}, is
\begin{displaymath}
	\begin{tikzpicture}[scale=0.7]
    
		\node (v0) at (0,0) {};
    \node (va) at (-2,2) {};
		\node (vc) at (2,2) {};
		\node (v1) at (0,4) {};

    \draw[fill] (v0)  circle (2pt) node [below] {$p_\varnothing$};
		\draw[fill] (va)  circle (2pt) node [left] {$p_\sfM$};
		\draw[fill] (vc)  circle (2pt) node [right] {$p_\sfN$};
		\draw[fill] (v1)  circle (2pt) node [above] {$p_0$};

		\path[-] (v0) edge  node [above] {} (va);
		\path[-] (v0) edge  node [above] {} (vc);
		\path[-] (va) edge  node [above] {} (v1);
		\path[-] (vc) edge  node [above] {} (v1);

  \end{tikzpicture}
\end{displaymath}
which has acquired two extra points and become both irreducible and local. The lattice of open subsets of $X = \fSpcnt\sfD^\mathrm{b}(k\times k)$ is
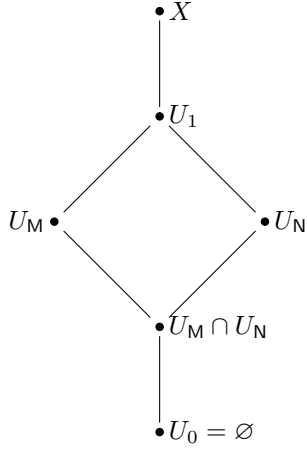
\begin{displaymath}
\begin{tikzpicture}[scale=0.7]
    
		\node (00) at (0,-2) {};
		\node (v0) at (0,0) {};
    \node (va) at (-2,2) {};
		\node (vc) at (2,2) {};
		\node (v1) at (0,4) {};
		\node (11) at (0,6) {};

    \draw[fill] (v0)  circle (2pt) node [right] {$U_\sfM\cap U_\sfN$};
		\draw[fill] (va)  circle (2pt) node [left] {$U_\sfM$};
		\draw[fill] (vc)  circle (2pt) node [right] {$U_\sfN$};
		\draw[fill] (v1)  circle (2pt) node [right] {$U_1$};
		\draw[fill] (00)  circle (2pt) node [right] {$U_0 = \varnothing$};
		\draw[fill] (11)  circle (2pt) node [right] {$X$};

		\path[-] (v0) edge  node [above] {} (va);
		\path[-] (v0) edge  node [above] {} (vc);
		\path[-] (va) edge  node [above] {} (v1);
		\path[-] (vc) edge  node [above] {} (v1);
		\path[-] (00) edge  node [above] {} (v0);
		\path[-] (11) edge  node [above] {} (v1);

  \end{tikzpicture}
\end{displaymath}
which contains a copy of the original lattice $\Thick(\sfD^\mathrm{b}(k\times k))$ via the embedding $l\mapsto U_l$.
\end{ex}

\subsection{Sobriety and a simplification}

We now further discuss the construction and analyse the space we produce; in particular, we see that the suggestive form of the examples we have treated so far is a general phenomenon. What we need works in complete generality, but is more enlightening and compelling with our construction in mind.

Let $L$ be a complete join semilattice and let $p\colon L\to \Sierp$ be a semipoint. Then, since $p$ is order preserving, $p^{-1}(0)$ is downward closed and, since $p$ preserves joins, $p^{-1}(0)$ is closed under all joins. Define $x\in L$ by $x = \vee p^{-1}(0)$. Then $x\in p^{-1}(0)$ by join closure, and for all $l\in p^{-1}(0)$ we have $l\leq x$. It follows that
\begin{displaymath}
p^{-1}(0) = [0,x] = \{l\in L\mid l\leq x\}.
\end{displaymath}
On the other hand, if $x\in L$ one easily checks that there is an associated semipoint given by the characteristic function for $L\setminus [0,x]$. 
This proves the following lemma.

\begin{lem}
There is a canonical identification of sets $L = \spt(L)$.
\end{lem}
We will use this identification without further mention and turn to discussing the topology. A subbase of opens for $\spt(L)$ is given by the subsets
\begin{displaymath}
U_l = \{p\in \spt(L) \mid p(l) = 1\}.
\end{displaymath}
Under the identification with $L$ we see that
\begin{align*}
U_l = \{x\in L\mid p_x(l) = 1\} &= \{x\in L\mid l\notin [0,x]\} \\
&= \{x\in L\mid l\nleq x\} \\
&= L\setminus [l,1].
\end{align*}

\begin{lem}
Under the identification $L = \spt(L)$ the closure of $l\in L$ is $[l,1]$.
\end{lem}
\begin{proof}
Suppose that $m\in \overline{\{l\}}$. This occurs if and only if for every closed subset $V$ of $L$ we have $l\in V$ implies $m\in V$, i.e.\ if for every open subset $U$ we have $m\in U$ implies $l\in U$. Thus, if $m\in U_l = L\setminus [l,1]$ we would have $l\in U_l$ which is nonsense. So $m\notin U_l$, i.e.\ $m\in [l,1]$.

On the other hand, suppose that $m\in [l,1]$. If $m\in U_r$, for $r\in L$, i.e.\ $m\notin [r,1]$ but $l\notin U_r$, i.e.\ $r\leq l$, then since we have assumed $l\leq m$ we obtain $r\leq m$ which is a contradiction. Thus if $m\in U_r$ we must have $l\in U_r$. Since every open is obtained from finite intersections of these subsets, it follows that if $m$ is in an open subset then $l$ must also lie in that open. Hence $m\in \overline{\{l\}}$.
\end{proof}

Combining what we have learned from our two lemmas we deduce the following result.

\begin{prop}\label{prop:L}
The space $\spt(L)$ is homeomorphic to $L$ with the topology whose closed subsets are finite unions of principal upward closed subsets of $L$, i.e.\ the subsets $[l,1]$ for $l\in L$. In particular, $\spt(L)$ is sober.
\end{prop}

From sobriety of $\spt(L)$ it follows that
\begin{displaymath}
\fSpcnt = \pt \circ \mcO \circ \spt \circ T^\op \cong \spt \circ T^\op
\end{displaymath}
i.e.\ in the context of our construction, for a triangulated category $\sfK$, $\fSpcnt(\sfK)$ is the set $T(\sfK) = \Thick(\sfK)$ with the topology described above. To reiterate: we have a canonical bijection between the points of $\fSpcnt(\sfK)$ and the elements of $\Thick(\sfK)$, and so $\fSpcnt(\sfK)$ is, in a weak sense, a coarse moduli space for the lattice of thick subcategories where the inclusion ordering of subcategories is recovered by the specialization ordering.

This is good news in the sense that we haven't lost information, but it means that computation in new cases is likely to be challenging\textemdash{}we have not made our lives any easier. However, one can optimistically expect that there are cases where, although we could not directly describe the elements of the lattice, one can construct the space by some other means. 

\subsection{The universal property}\label{ssec:universal}

As $\fSpcnt$ is freely constructed it satisfies a universal property: the unit of adjunction $T(\sfK) \to i\mcO\fSpcnt(\sfK)$ is the initial map from $T(\sfK)$ to a spatial frame, i.e.\ any join preserving map from $T(\sfK)$ to a spatial frame factors via the lattice of open subsets of $\fSpcnt(\sfK)$. 

Let us describe how this can be converted to a support theoretic universal property in line with Balmer's original work. It will be more convenient for us to work with open subsets, rather than Thomason subsets, and so our support theories will associate open, rather than closed, subsets to objects. This is, from the point of view of tt-geometry, purely cosmetic and corresponds to working with the Hochster dual of the Balmer spectrum.

Let $X$ be a Sober space with corresponding spatial frame $\mcO(X)$. 

\begin{defn}
A $\vee$-support datum on $\sfK$ with values in $X$ is a map $\sigma\colon \Ob\sfK \to \mcO(X)$ from objects of $\sfK$ to open subsets of $X$ such that for all $x,y\in \sfK$:
\begin{itemize}
\item[(i)] $\sigma(0) = \varnothing$;
\item[(ii)] $\sigma(\Sigma^i x) = \sigma(x)$ for all $i\in \ZZ$;
\item[(iii)] $\sigma(x\oplus y) = \sigma(x)\cup \sigma(y)$;
\item[(iv)] for every triangle $x\to y\to z$ we have $\sigma(y)\subseteq \sigma(x)\cup \sigma(z)$.
\end{itemize}
\end{defn}

A support datum's purpose in life, whether it be of the kind described above or as in the usual tt-setting, is to describe a lattice map from some lattice of thick subcategories to $\mcO(X)$.

\begin{prop}\label{prop:support}
Giving a $\vee$-support datum on $\sfK$ with values in $X$ is equivalent to specifying a morphism $T(\sfK) \to \mcO(X)$ in $\CjSLat$.
\end{prop}
\begin{proof}
Suppose we are given a $\vee$-support datum $\sigma$ on $\sfK$ with values in $X$. We define a morphism $\phi\colon T(\sfK) \to \mcO(X)$ by setting
\[
\phi(\sfM) = \bigcup_{x\in \sfM} \sigma(x).
\]
This is clearly an open subset of $X$ and $\phi$ is order preserving by definition. Condition (i) that $\sigma(0)=\varnothing$ tells us that $\phi$ preserves the bottom element, i.e.\ the empty join. Suppose then that $\sfM_i$ for $i\in I$ is a family of thick subcategories with join $\sfM$. We need to show equality of
\[
U = \phi( \sfM) \text{ and } V = \cup_i \phi(\sfM_i).
\]
Since each $\sfM_i$ is contained in $\sfM$ it is immediate from the construction that $V\subseteq U$. On the other hand given $x\in \sfM$ we can find, by Lemma~\ref{lem:finite}, $x_1,\ldots,x_n \in \cup_i \sfM_i$ such that $x\in \thick(x_1,\ldots,x_n)$. It follows from axioms (ii)-(iv) that
\[
\sigma(x) \subseteq \sigma(x_1)\cup \cdots \cup \sigma(x_n) \subseteq V
\]
showing that $U\subseteq V$.

Given a morphism $\phi\colon T(\sfK) \to \mcO(X)$ in $\CjSLat$ we define
\[
\sigma(x) = \phi(\thick(x)).
\]
This satisfies (ii) by construction and (i),(iii), and (iv) follow from the fact that $\phi$ preserves joins.
\end{proof}

There is a natural $\vee$-support datum on $\sfK$ taking values in $\fSpcnt(\sfK)$ which is given by sending $x$ to the open $U_{\thick(x)}$ corresponding to $\thick(x)$, and this is also universal in the appropriate sense. Given a $\vee$-support datum $\sigma$ on $\sfK$ taking values in $X$ there is a corresponding morphism $\phi\colon T(\sfK) \to i\mcO(X)$ in $\CjSLat$. The map $\phi$ corresponds, via adjunction, to a frame map $\lefta T(\sfK) \to \mcO(X)$ and we can take points to obtain $X \to \fSpcnt(\sfK)$ such that $\sigma$ is obtained by pulling back the universal support datum. On the other hand, given $X \to \fSpcnt(\sfK)$ we can take this pullback, i.e.\ consider the composite
\[
T(\sfK) \to i\mcO(\fSpcnt(\sfK)) \to i\mcO(X)
\]
which is the join-preserving map associated to the corresponding $\vee$-support datum. 



\section{The partially functorial theory}\label{sec:pf}
\setcounter{subsection}{1}

As seen in Section~\ref{ssec:fail} if one works with all functors, and hence complete join semilattices, one cannot reconstruct the Balmer spectrum. In this section we develop a partially functorial theory based on $\CLat$ and $\lefter$ which addresses this issue by giving the correct result for categories where the lattice of thick subcategories is distributive.

We have already seen in Example~\ref{ex:letsnotmeet} that $T\colon \Tcat \to \CjSLat$ does not factor via $\CLat$, i.e.\ we may not get meet preserving maps out of $T$. The situation isn't improved in the case that the lattices involved are distributive.

\begin{ex}
Let $k$ be a field and consider the diagonal inclusion $f\colon k\to k\times k = R$ and the corresponding bounded derived categories. We know that all lattices of thick subcategories involved are coherent frames: the lattices are $\Sierp$ and $\Sierp \times \Sierp$ (i.e.\ the powerset of $\{0,1\}$) respectively. However, the functor $f_*\colon \sfD^\mathrm{perf}(R) \to \sfD^\mathrm{perf}(k)$ does not induce a map of frames, i.e.\ does not give a map in $\CLat$. 

Indeed, writing $R = R_1\times R_2$ we have
\begin{displaymath}
\thick(f_*(\thick(R_1)\cap \thick(R_2))) = \thick(f_*0) = 0
\end{displaymath}
but, on the other hand,
\begin{displaymath}
\thick(f_*R_1) \cap \thick(f_*R_2) = \sfD^\mathrm{perf}(k) \cap \sfD^\mathrm{perf}(k) = \sfD^\mathrm{perf}(k)
\end{displaymath}
exhibiting that $T(f_*)$ fails to preserve meets. 
\end{ex}

We see that even with distributivity hypotheses arbitrary exact functors may fail to induce maps in $\CLat$. So, if we wish to work in the category of complete lattices we must restrict the functors we allow.

\begin{defn}\label{defn:confluent}
Let $F\colon \sfK \to \sfL$ be a morphism in $\Tcat$, i.e.\ an exact functor between essentially small triangulated categories. We say that $F$ is a \emph{confluent functor} if $T(F)$ preserves finite meets, i.e.\ $F(\sfK)$ generates $\sfL$ and 
\begin{displaymath}
\thick(F\sfM)\cap \thick(F\sfN) = \thick(F(\sfM\cap \sfN)).
\end{displaymath}
for every pair of thick subcategories $\sfM$ and $\sfN$ of $\sfK$.
\end{defn}

\begin{rem}
The condition that $F(\sfK)$ generates $\sfL$ is forced by the fact that the empty meet, namely the top element, needs to be preserved by $T(F)$. The analogous constraint in tt-geometry is not serious, the tensor unit always generates as a tensor ideal, but here it is a restriction.
\end{rem}

\begin{rem}
We will sometimes say that an exact functor $F\colon \sfK \to \sfL$ is confluent with respect to some class $\mcS$ of thick subcategories of $\sfK$ (e.g.\ radical tensor ideals) if $T(F)$ preserves meets of members of $\mcS$.
\end{rem}

\begin{rem}
We note that the containment 
\begin{displaymath}
\thick(F\sfM)\cap \thick(F\sfN) \supseteq \thick(F(\sfM\cap \sfN)).
\end{displaymath}
always holds.
\end{rem}

\begin{ex}
Let $f\colon A\to B$ be a map of commutative rings. The induced functor $f^*\colon \sfD^\mathrm{perf}(A) \to \sfD^\mathrm{perf}(B)$ is confluent. The generation condition is clear as $f^*A \cong B$, and preservation of meets follows from the classification of thick subcategories via support, as in \cite{Thomclass}, and the formula, for $E \in \sfD^\mathrm{perf}(A)$, 
\[
\supp f^*E = \Spec(f)^{-1} \supp E
\]
(see Lemma~\ref{lem:ttconfluent} for details).
\end{ex}

\begin{ex}\label{ex:notok}
Localizations are not necessarily confluent. Consider the localization
\[
\pi\colon \sfD^\mathrm{b}(\PP^1) \to \sfD^\mathrm{b}(\AA^1)
\]
at $\thick(k(x))$ for some point $x\in \PP^1$. Then 
\[
\thick(\pi(\thick(\mcO) \cap \thick(\mcO(1)))) = 0 \text{ but } \thick(\pi\mcO) \cap \thick(\pi \mcO(1)) = \sfD^\mathrm{b}(\AA^1)
\]
showing that even passing to an open subset of the Balmer spectrum can behave poorly with respect to the whole lattice of thick subcategories.
\end{ex}


\begin{lem}\label{lem:subcat}
Let $\sfK \xrightarrow{F} \sfL \xrightarrow{G} \sfM$ be exact functors. Then:
\begin{itemize}
\item[(a)] if $F$ is an equivalence it is confluent;
\item[(b)] if $F$ and $G$ are confluent so is $GF$.
\end{itemize}
\end{lem}
\begin{proof}
We begin with (a): if $F$ is an equivalence then, for thick subcategories $\sfN_1$ and $\sfN_2$ of $\sfK$, we have
\begin{align*}
\thick(F(\sfN_1)) \cap \thick(F(\sfN_2)) &= F(\sfN_1) \cap F(\sfN_2) \\
&= F(\sfN_1 \cap \sfN_2) \\
&= \thick(F(\sfN_1\cap \sfN_2)).
\end{align*}
Here the first and last equalities are the fact that $F$ is an equivalence and so sends thick subcategories to thick subcategories, and the middle equality is that $F$ is an equivalence and so reflects isomorphisms. There is a small lie here as strictly we still need to close under isomorphisms i.e.\ $F\sfN_1$ may fail to be closed under isomorphisms, but this can be accomplished without harm or we can assume $\sfK$ and $\sfL$ are skeletal without changing the associated lattices. It is clear that $F(\sfK)$ generates $\sfL$.

Now for (b), suppose that $F$ and $G$ are confluent. We know that $F(\sfK)$ generates $\sfL$, i.e.\ $\thick(F(\sfK)) = \sfL$. By Lemma~\ref{lem:commute} and confluence of $G$ we have
\begin{displaymath}
\sfM = \thick(G(\sfL)) = \thick(G\thick(F(\sfK))) = \thick(GF\sfK).
\end{displaymath}
Next let us turn to non-empty meets. We see that, for thick subcategories $\sfN_1$ and $\sfN_2$ of $\sfK$,
\begin{align*}
\thick(GF(\sfN_1)) \cap \thick(GF(\sfN_2)) &= \thick(G(\thick(F\sfN_1))) \cap \thick(G(\thick(F\sfN_2))) \\
&= \thick(G(\thick(F\sfN_1) \cap \thick(F\sfN_2))) \\
&= \thick(GF(\sfN_1 \cap \sfN_2))
\end{align*}
where the first equality is Lemma~\ref{lem:commute}, the second is that $G$ is confluent, and the third is that $F$ is confluent (and another application of Lemma~\ref{lem:commute}).
\end{proof}

\begin{defn}\label{defn:htcat}
By the Lemma, taking $\hTcat$ to be the collection of essentially small triangulated categories with confluent exact functors between them gives a subcategory of $\Tcat$ which contains all equivalences.
\end{defn}

By construction the functor $T\colon \hTcat \to \CjSLat$ factors via $T_\wedge \colon \hTcat \to \CLat$.

\begin{defn}\label{defn:spcnt}
We define the \emph{non-tensor spectrum} functor $\Spcnt\colon \hTcat^\op \to \Sob$ to be the composite
\begin{displaymath}
\hTcat^\op \xrightarrow{T_\wedge^\op} \CLat^\op \xrightarrow{\lefter^\op} \SFrm^\op \xrightarrow{\Spec} \Sob
\end{displaymath}
where $\lefter$ is the adjoint of Proposition~\ref{prop:adjoint2}.
\end{defn}

\begin{rem}
We note that $\Spcnt$, like $\fSpcnt$, can be applied to \emph{any} essentially small triangulated category; the only issue is functoriality.
\end{rem}

We can actually simplify the construction of $\Spcnt$, as we did for $\fSpcnt$. We have noted in Lemma~\ref{lem:sober} that for any complete lattice $L$ the space $\CLat(L,\Sierp)$ is sober. Hence
\begin{displaymath}
\Spec\lefter L = \Spec \mcO \CLat(L,\Sierp) \cong \CLat(L,\Sierp)
\end{displaymath}
by Stone duality. Thus $\Spcnt(\sfK) \cong \CLat(T(\sfK),\Sierp)$.

The non-tensor spectrum satisfies a universal property, which can be neatly expressed as follows.

\begin{lem}\label{lem:universal}
There is a natural isomorphism $\mcO\Spcnt \cong j\lefter$. Thus there is a natural transformation $T_\wedge \to \mcO\Spcnt$ exhibiting $\mcO\Spcnt$ as the spatial frame approximation $\lefter T_\wedge$ (up to taking the opposite functors) of $T_\wedge$ in $\CLat$.
\end{lem}
\begin{proof}
The composite $\mcO \Spcnt$ is, by definition, 
\begin{displaymath}
\mcO \circ \Spec \circ\; \lefter^\op \circ T_\wedge^\op \cong \lefter^\op \circ T_\wedge^\op.
\end{displaymath}
\end{proof}

By Stone duality one can reinterpret this property in terms of continuous maps of spaces. However, despite being `obvious' i.e.\ given by intersections, the meets in $T(\sfK)$ are somewhat subtle: we simply do not know how to produce a generator for the intersection $\thick(x)\cap \thick(y)$, if it even has one! Without a tensor product there is no obvious operation on objects which allows us to define a map $T_\wedge(\sfK) \to F$ in $\CLat$, with $F$ a spatial frame, purely in terms of objects of $\sfK$ as in Section~\ref{ssec:universal} or tt-geometry.

By construction $\Spcnt$ doesn't lose any information in the case that $T(\sfK)$ is distributive.

\begin{thm}\label{thm:class}
Suppose that $\sfK$ is such that $T(\sfK)$ is distributive. Then the comparison map is a lattice isomorphism
\begin{displaymath}
T(\sfK) \xrightarrow{\sim} \mcO\Spcnt \sfK,
\end{displaymath}
i.e.\ thick subcategories of $\sfK$ are classified by $\Spcnt \sfK$. The inverse is given by sending an open subset $U$ to the thick subcategory
\begin{displaymath}
\bigcap_{\sfP\notin U} \sfP
\end{displaymath}
where we have identified $\Spcnt \sfK$ with the space of meet-prime thick subcategories. 
\end{thm}
\begin{proof}
By Corollary~\ref{cor:spatial} the lattice $T(\sfK)$ is a spatial frame. Hence $\lefter T(\sfK)\cong T(\sfK)$, and so the claimed lattice isomorphism is immediate from Lemma~\ref{lem:universal}. The explicit description of the inverse is a standard fact, and can easily be verified directly.
\end{proof}

\begin{rem}
The classification mirrors the one from tt-geometry, cf.\ \cite{BaSpec}*{Lemma~4.8} in particular for a similar description of the inverse.
\end{rem}


\subsection{The comparison with tt-geometry}

We now discuss the relation to tt-geometry in the case of tt-categories. Given a tt-category $\sfT$ we denote by $T^\otimes(\sfT)$ the lattice of radical thick $\otimes$-ideals, and for a tt-functor $F\colon \sfT\to \sfS$ we denote by $T^\otimes(F)$ the corresponding map $T^\otimes(\sfT)\to T^\otimes(\sfS)$ sending $\sfJ$ to the smallest radical thick $\otimes$-ideal containing $F\sfJ$. As usual $\Spc \sfT$ denotes the Balmer spectrum i.e.\ the space $(\Spec T^\otimes(\sfK))^\vee$. We know $T^\otimes(\sfT)$ is a coherent frame. Moreover, $T^\otimes(F)$ is a map of frames preserving compact objects;  we give a proof of the first assertion in the next lemma.

\begin{lem}\label{lem:ttconfluent}
Let $\sfT$ and $\sfS$ be tt-categories and $F\colon \sfT \to \sfS$ be a tt-functor. Then $F$ is confluent with respect to radical thick $\otimes$-ideals, i.e.\ $T^\otimes(F)$ preserves finite meets.
\end{lem}
\begin{proof}
Suppose that $\sfI$ and $\sfJ$ are radical $\otimes$-ideals of $\sfT$. Then, by \cite{BaSpec}*{Theorem~4.10}, we have $\sfI = \sfT_V$ and $\sfJ = \sfT_W$, the ideals of objects supported on $V$ and $W$ respectively, where $V$ and $W$ are Thomason subsets of $\Spc \sfT$. By the classification and \cite{BaSpec}*{Proposition~3.6} we have
\begin{displaymath}
T^\otimes(F)\sfI = \sfS_{\supp F\sfI} = \sfS_{\Spc(F)^{-1}V} \text{ and } T^\otimes(F)\sfJ = \sfS_{\supp F\sfJ} = \sfS_{\Spc(F)^{-1}W}.
\end{displaymath}
Thus we see, using that $\sfI\cap\sfJ = \sfT_{V\cap W}$, 
\begin{displaymath}
T^\otimes(F)\sfI \cap T^\otimes(F)\sfJ = \sfS_{\Spc(F)^{-1}V \cap \Spc(F)^{-1}W} = \sfS_{\Spc(F)^{-1}(V\cap W)} = T^\otimes(F)(\sfI\cap \sfJ).
\end{displaymath}
It is clear that $T^\otimes(F)(\sfT) = \sfS$.
\end{proof}

Thus we get a functor $T^\otimes\colon \ttcat \to \CohFrm$, from the category of tt-categories and tt-functors to the category of coherent frames with maps preserving joins, finite meets, and compacts, which gives Balmer's theory of tt-geometry by composing with Stone duality (up to some pesky $\op$'s which we suppress, see Theorem~\ref{thm:comparisonB} for a precise statement).

We do \emph{not} get a functor $T\colon \ttcat \to \CLat$ by sending $\sfT$ to its lattice of thick subcategories; a tt-functor may not respect meets of general thick subcategories, cf.\ Example~\ref{ex:notok}.

However, we are able to produce a natural transformation between the restrictions of $T$ and $T^\otimes$ on a subcategory of $\ttcat$.

\begin{lem}\label{lem:Bcomparison}
Let $\sfT$ be a rigid tt-category. Then $T^\otimes(\sfT)$ is a bounded sublattice of $T(\sfT)$.
\end{lem}
\begin{proof}
Both are collections of subcategories ordered by inclusion, so $T^\otimes(\sfT)$ is a subposet of $T(\sfT)$. Clearly the maximal elements coincide and so do the minimal ones\textemdash{}the key point is that rigidity of $\sfT$ implies that every $\otimes$-ideal is radical and so the zero ideal is radical.

It remains to show that $T^\otimes(\sfT)$ is closed under meets and joins in $T(\sfT)$. The former is clear: the intersection of radical $\otimes$-ideals is again a radical $\otimes$-ideal, and so $T^\otimes(\sfT)$ is closed under arbitrary meets (and this doesn't require rigidity).

For the latter, let $\sfJ_i$ for $i\in I$ be a collection of radical $\otimes$-ideals and consider
\begin{displaymath}
\sfX = \bigcup_{i\in I}\sfJ_i \text{ and } \sfJ = \thick(\sfX),
\end{displaymath}
where $\sfJ$ is the join of the $\sfJ_i$ in $T(\sfT)$. We observe that $\sfX$ is closed under the tensor product, and so $\sfJ$ is a $\otimes$-ideal (see for instance \cite{StevensonActions}*{Lemma~3.8}). As $\sfT$ is rigid $\sfJ$ is automatically radical, and so $\sfJ\in T^\otimes(\sfT)$. Therefore $\sfJ$ must be the join in $T^\otimes(\sfT)$ of the $\sfJ_i$ as required.
\end{proof}

\begin{rem}
The lemma does not hold in complete generality: $\otimes$-nilpotent objects will result in the minimal element $0$ of $T(\sfT)$ failing to be a radical $\otimes$-ideal, and so the inclusion will not preserve the empty join.
\end{rem}

As Example~\ref{ex:notok} illustrates the correct setting in which to compare $T$ and $T^\otimes$ is a delicate issue, at least if we want a natural transformation rather than just an objectwise comparison. The following is probably not optimal, but is sufficient to illustrate the situation. We denote by $\rttcat_\wedge^{\mathrm{loc}}$ the category of rigid tt-categories with exact monoidal localizations which are confluent, e.g.\ the monoidal equivalences.

\begin{lem}\label{lem:naturalB}
The inclusions of Lemma~\ref{lem:Bcomparison} assemble to give a natural transformation $T^\otimes \to T_\wedge$ of functors $\rttcat_\wedge^{\mathrm{loc}} \to \CLat$. 
\end{lem}
\begin{proof}
Suppose that $F\colon \sfT\to \sfS$ is a confluent monoidal localization and consider the corresponding square
\begin{displaymath}
\xymatrix{
T^\otimes(\sfT) \ar[r] \ar[d] & T^\otimes(\sfS) \ar[d] \\
T(\sfT) \ar[r] & T(\sfS)
}
\end{displaymath}
Let $\sfJ$ be a thick $\otimes$-ideal of $\sfT$. Then, because $F$ is essentially surjective, $F\sfJ$ is closed under the tensor product in $\sfS$ and so $\thick(F\sfJ)$ is a $\otimes$-ideal of $\sfS$, and hence the smallest $\otimes$-ideal of $\sfS$ containing $F\sfJ$. This shows
\begin{displaymath}
T^\otimes(F)(\sfJ) = \thick^\otimes(F\sfJ) = \thick(F\sfJ) = T(F)(\sfJ)
\end{displaymath}
proving the square commutes.
\end{proof}  

Let $X$ be a spectral space. Recall from Definition~\ref{defn:Hochsterdual} the Hochster dual of $X$, which we denote by $X^\vee$. 

\begin{thm}\label{thm:comparisonB}
There is a canonical natural transformation $\varphi\colon \Spcnt \to \Spc^\vee$ of the functors $(\rttcat_\wedge^{\mathrm{loc}})^\op \to \Sob$, i.e. for every rigid tt-category $\sfT$ there is a canonical comparison map 
\begin{displaymath}
\varphi_\sfT \colon \Spcnt\sfT \to (\Spc \sfT)^\vee
\end{displaymath}
which is natural with respect to confluent monoidal localizations. 
\end{thm}
\begin{proof}
The natural transformation of Lemma~\ref{lem:naturalB} can be postcomposed with $\Spec$ to yield a natural transformation
\begin{displaymath}
\Spcnt \to \Spec(T^\otimes(-)).
\end{displaymath}
The subtlety is that $\Spec(T^\otimes(-))$ is not quite $\Spc$. Identifying $\Spec(T^\otimes(-))$ with the set of prime $\otimes$-ideals of $\sfT$ one checks that the closure of $\sfP$ is $Z_\sfP = \{\sfQ\mid \sfP\subseteq \sfQ\}$, rather than the Balmer-closed subset $\{\sfQ\mid \sfQ\subseteq \sfP\}$. This is rectified by reversing the order, i.e.\ taking the Hochster dual.
\end{proof}

\begin{rem}
To give another perspective on the appearence of Hochster duality in the theorem, let us note that one recovers $T^\otimes$ from $\Spc$ via the lattice of Thomason subsets whereas via Stone duality we would recover $T^\otimes$ via the lattice of open subsets. The Thomason subsets of $\Spc\sfT$ are precisely the open subsets of $\Spc\sfT^\vee$.
\end{rem}

\begin{rem}
In practice it is convenient to have $\Spc\sfD^\mathrm{perf}(R)\cong \Spec R$ and so we would like access to ``$\varphi^\vee\colon \Spcnt^\vee \to \Spc$''. However, as $\Spcnt$ need not, at least \emph{a priori}, be spectral we do not necessarily have access to Hochster duality for the non-tensor spectrum.
\end{rem}

We defer the discussion of examples to Section~\ref{sec:ex}.


\subsection{The comparison with Matsui's work}

There are other theories on the market which should be compared to our offering. We begin by discussing the work of Matsui \cite{matsui2019}. Matsui introduces a notion of support theory, based on the work of Balmer \cite{BaSpec}, and shows it classifies thick subcategories under suitable noetherian, sobriety, and distributivity hypotheses (the last of which is left implicit). As we now show the non-tensor spectrum generalizes Matsui's machinery under strong (in a vacuum) but standard hypotheses.

Let us begin by reinterpreting part of the setup of \cite{matsui2019}. Fix some triangulated category $\sfK$. Let $\MSpec \sfK$ ($M$ is for Matsui) denote the collection of finitely generated join-prime elements of $T(\sfK)$, and define for each $x\in \sfK$ a basic closed subset
\begin{displaymath}
Z_x = \{\sfJ \in \MSpec \sfK \mid \sfJ \subseteq \thick(x)\},
\end{displaymath}
and topologize $\MSpec \sfK$ by letting these determine the closed subsets. 

In general we have the following statement, expressing Matsui's construction as dual to ours when join-prime thick subcategories are under control.

\begin{prop}\label{prop:matsui}
Suppose that every join-prime thick subcategory of $\sfK$ is finitely generated. Then
\begin{displaymath}
\MSpec \sfK = \Spec (\lefter (T(\sfK)^\op)).
\end{displaymath}
\end{prop}
\begin{proof}
A join-prime in $T(\sfK)$ is precisely a meet-prime in $T(\sfK)^\op$ and passing from $T(\sfK)$ to $T(\sfK)^\op$ sends our closed subsets to Matsui's closed subsets. Thus $\MSpec \sfK$ is the collection of meet-prime elements of $T(\sfK)^\op$ together with the natural topology, or equivalently the spectrum of the spatial frame approximation $\lefter (T(\sfK)^\op)$. 
\end{proof}

\begin{rem}
We note that while join-primes are often finitely generated there is an asymmetry: in general meet-prime thick subcategories will \emph{not} be finitely generated. This is already the case for $\sfD^\mathrm{b}(\ZZ)$.
\end{rem}

We have the following theorem which covers the main examples of \cite{matsui2019}, for instance, tt-categories with noetherian spectrum where the unit generates (cf.\ Lemma~\ref{lem:unit}) and singularity categories of complete intersections.

\begin{thm}\label{thm:comparisonM}
If $T(\sfK)$ is a coherent frame and $\Spcnt\sfK^\vee$ is noetherian, then $\MSpec \sfK \cong \Spcnt\sfK^\vee$. 
\end{thm}
\begin{proof}
Suppose that $T(\sfK)$ is a coherent frame such that $\Spcnt\sfK^\vee$ is noetherian. Then, by Stone duality (see Theorem~\ref{thm:class}), $\Spcnt\sfK^\vee$ gives a classifying support datum for $\sfK$ in the sense of \cite{matsui2019}*{Definition~2.9}. Mastui's Theorem~2.13 thus provides the claimed homeomorphism.
\end{proof}

\begin{rem}
The Theorem, when combined with our Theorem~\ref{thm:comparisonB} and discussion of functoriality, complements \cite{matsui2019} by expressing part of the functoriality of the construction and making explicit the general comparison with Balmer's theory.
\end{rem}

\subsection{The comparison with the ``noncommutative tt-spectrum''}

Finally, we discuss a potential lattice theoretic approach to \cite{nakano2019} and the corresponding comparison with the non-tensor spectrum. Let $(\sfR,\otimes,\unit)$ be a monoidal triangulated category, such that the tensor product is exact in each variable. We do not assume that $\sfR$ is symmetric, or even braided, monoidal. By $\otimes$-ideal in this section we will always mean $2$-sided $\otimes$-ideal.

Following noncommutative ring theory and \cite{nakano2019} we say that a 2-sided thick $\otimes$-ideal $\sfJ$ of $\sfR$ is \emph{semiprime} if given $r\in \sfR$ then $r\otimes s \otimes r\in \sfJ$ for all $s\in \sfR$ implies that $r\in \sfJ$. Given a 2-sided $\otimes$-ideal $\sfI$ we denote by $\sqrt{\sfI}$ the smallest semiprime $\otimes$-ideal containing $\sfI$. Let us denote by $\sfT^\otimes(\sfR)$ the complete lattice of semiprime $\otimes$-ideals of $\sfR$, where the ordering is via inclusion and arbitrary meets are given by intersection. This is not in conflict with our earlier notation, as semiprime reduces to radical in the symmetric setting.

A proper 2-sided $\otimes$-ideal $\sfP$ is \emph{prime} in the sense of \cite{nakano2019} if for 2-sided $\otimes$-ideals $\sfI$ and $\sfJ$ we have $\sfI\otimes \sfJ \subseteq \sfP$ implies $\sfI\subseteq \sfP$ or $\sfJ\subseteq \sfP$. 

\begin{rem}
In \cite{nakano2019} a semiprime $\otimes$-ideal is defined to be an intersection of prime $\otimes$-ideals and their Theorem~3.4.2 shows the equivalence of that definition with the one given above. In particular, it is not clear from the definitions above that prime $\otimes$-ideals are semiprime, but they are.
\end{rem}

For us, the point is that this notion of prime is related to the corresponding lattice theoretic notion.

\begin{lem}\label{lem:rad}
Let $\sfI,\sfJ\in T^\otimes(\sfR)$. Then
\begin{displaymath}
\sqrt{\thick(\sfI\otimes \sfJ)} = \sfI \cap \sfJ = \sqrt{\thick(\sfJ\otimes \sfI)}.
\end{displaymath}
\end{lem}
\begin{proof}
By symmetry it is enough to prove the equality on the left. It is clear, by the ideal property, that $\sfI\otimes \sfJ$ is contained in $\sfI\cap \sfJ$. Since $\sfI\cap \sfJ$ is semiprime the left-hand side is contained in it.

Suppose, on the other hand, that $j\in \sfI\cap \sfJ$. Then, for any $r\in \sfR$, we have $j\otimes r\otimes  j \in \sfI\otimes \sfJ$. Hence $j\in \sqrt{\thick(\sfI\otimes \sfJ)}$ proving the claimed equality. 
\end{proof}

We will say a $2$-sided $\otimes$-ideal $\sfP$ is \emph{prime with respect to semiprime $\otimes$-ideals} if it is semiprime, and for semiprime $\otimes$-ideals $\sfI$ and $\sfJ$ we have $\sfI\otimes \sfJ \subseteq \sfP$ implies one of $\sfI$ or $\sfJ$ is contained in $\sfP$.

\begin{lem}\label{lem:sameprimes}
A 2-sided $\otimes$-ideal $\sfP$ is prime with respect to semiprime $\otimes$-ideals if and only if it is meet-prime in $T^\otimes(\sfR)$.
\end{lem}
\begin{proof}

Suppose that $\sfP$ is semiprime. Given semiprime $\otimes$-ideals $\sfI$ and $\sfJ$ we have $\sfI \otimes \sfJ \subseteq \sfP$ if and only if $\sqrt{\thick(\sfI\otimes \sfJ)}$ is contained in $\sfP$, by virtue of $\sfP$ being semiprime and thick, i.e.\ if and only if $\sfI\cap \sfJ \subseteq \sfP$. Thus asking for $\sfI\otimes \sfJ$ contained in $\sfP$ to imply one of $\sfI$ or $\sfJ$ is contained in $\sfP$ is the same as asking for $\sfI\cap \sfJ$ to imply the same. 
\end{proof}

\begin{rem}
In the symmetric setting we have a good description of the radical and one can deduce that being prime with respect to semiprime $\otimes$-ideals is equivalent to being prime. However, in the noncommutative setting this is missing.
\end{rem}

The dual topology on $\CLat(T^\otimes(-),\Sierp)$, which we denote by $D(\CLat(T^\otimes(-),\Sierp))$ is given by declaring the closed subsets to be generated by the quasi-compact open subsets of $\CLat(T^\otimes(-),\Sierp)$. We note that if our space is spectral this is precisely Hochster duality. 

\begin{thm}\label{thm:comparisonN}
Suppose that every $\otimes$-ideal of $\sfR$ is semiprime. Then the noncommutative Balmer spectrum $\Spc^\mathrm{nc}$ of \cite{nakano2019} is given by
\begin{displaymath}
D\Spec (\lefter T^\otimes(\sfR)) = D\CLat(T^\otimes(\sfR),\Sierp).
\end{displaymath}
\end{thm}
\begin{proof}
By Lemma~\ref{lem:sameprimes} the underlying set of $\Spc^{\mathrm{nc}}$ is precisely $\CLat(T^\otimes(-),\Sierp)$, and unwinding the definitions one sees that (as in the case of the Balmer spectrum) the topology is the dual one.
\end{proof}

\begin{rem}
By \cite{Nakano21}*{Proposition~4.1.1} if every object of $\sfR$ is left or right dualizable then every $\otimes$-ideal is semiprime and one can apply the theorem.
\end{rem}

With the assumption of Theorem~\ref{thm:comparisonN} $T^\otimes(\sfR)$ is a sublattice of $T(\sfR)$ and so there is a comparison map
\[
\Spcnt \sfR \to \CLat(T^\otimes(\sfR),\Sierp).
\]
However, as noted the topology of $\CLat(T^\otimes(\sfR),\Sierp)$ does not necessarily agree with that of $\Spc^\mathrm{nc}(\sfR)$. This can be compensated for in various cases. For instance, if $T^\otimes(\sfR) \to T(\sfR)$ preserves finite presentation, i.e.\ finite generation as an ideal implies finite generation as a thick subcategory, then there is an induced map
\[
D(\Spcnt \sfR) \to D\CLat(T^\otimes(\sfR),\Sierp) = \Spc^\mathrm{nc} \sfR,
\]
and in the case that $\Spc^\mathrm{nc} \sfR$ is spectral then we get a comparison map
\[
\Spcnt \sfR \to (\Spc^\mathrm{nc} \sfR)^\vee.
\]

\begin{rem}
The lattice-theoretic point of view on the noncommutative spectrum has also been considered in \cite{Mallick}. Mallick and Ray show that, provided every prime ideal is completely prime, the lattice of semiprime thick tensor ideals forms a coherent frame.
\end{rem}



\section{Examples}\label{sec:ex}
\setcounter{subsection}{1}

In this short section we discuss a few additional examples, which are the standard ones in various contexts, to illustrate the machinery.

\begin{ex}
Let us consider $\sfK = \sfD^\mathrm{perf}(R)$ for a commutative ring $R$. In this case, since the tensor unit $R$ generates and $\sfK$ is rigid, $T(\sfK) = T^\otimes(\sfK)$ is a coherent frame (cf.\ Lemma~\ref{lem:unit}). We have 
\[
\Spcnt \sfK \cong (\Spec R)^\vee
\]
due to Theorem~\ref{thm:comparisonB}. Indeed, since $T(\sfK)$ is already a coherent and hence spatial frame, we have $\lefter T(\sfK) = T(\sfK)$ and so
\[
\Spcnt \sfK = \Spec \lefter T (\sfK) \cong \Spec T(\sfK) = (\Spc \sfK)^\vee \cong (\Spec R)^\vee.
\] 
We know from Proposition~\ref{prop:L} that $\fSpcnt \sfK$ is $T(\sfK)$ with the topology generated by the opens $U_\sfM$ for $\sfM\in T(\sfK)$ a thick subcategory. There is a comparison map
\[
\Spcnt \sfK \to \fSpcnt\sfK, \quad \sfP \mapsto \sfP
\]
where we identify $\Spcnt \sfK$ with the set of meet-prime thick subcategories.
\end{ex}

\begin{rem}\label{rem:SvsfS}
Given any complete lattice $L$ there is an inclusion
\[
\CLat(L, \Sierp) \to \CjSLat(L, \Sierp)
\]
which is a continuous map with respect to the corresponding topologies. Thus there is always a comparison map $\Spcnt \sfK \to \fSpcnt \sfK$ as described in the above example, and using the identification of Proposition~\ref{prop:L} this map is just the inclusion of the meet-prime elements of $L$ into $L$. 
\end{rem}

\begin{ex}
Suppose that $(R,\mfm,k)$ is an isolated local complete intersection of codimension $c$, and let $\sfK = \sfD_\mathrm{sg}(R)$ be the singularity category. It is proved in \cite{Stevensonclass} that $T(\sfK)$ is the coherent frame of Thomason subsets of $\PP^{c-1}_k$. Thus $(\Spcnt \sfK)^\vee \cong \PP^{c-1}_k$ as topological spaces. 
\end{ex}

\begin{ex}\label{ex:pp1}
Consider $\sfK = \sfD^\mathrm{b}(\PP^1_k)$, the bounded derived category of coherent sheaves on $\PP^1_k$ for some field $k$, whose corresponding lattice $T(\sfK)$ was discussed in Example~\ref{ex:notmodular}. We claim that $\Spcnt \sfK = \pt T(\sfK) = \varnothing$. Indeed, given distinct integers $i,j,$ and $k$ the lattice $T(\sfK)$ contains a bounded sublattice
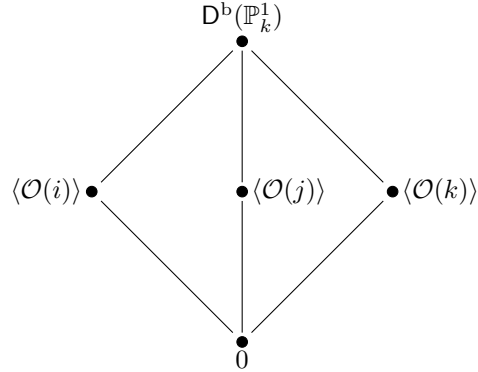
\begin{displaymath}
	\begin{tikzpicture}
    
		\node (v0) at (0,0) {};
    \node (va) at (-2,2) {};
    \node (vb) at (0,2) {};
		\node (vc) at (2,2) {};
		\node (v1) at (0,4) {};
    
    \draw[fill] (v0)  circle (2pt) node [below] {0};
		\draw[fill] (va)  circle (2pt) node [left] {$\langle\mcO(i)\rangle$};
		\draw[fill] (vb)  circle (2pt) node [right] {$\langle\mcO(j)\rangle$};
		\draw[fill] (vc)  circle (2pt) node [right] {$\langle\mcO(k)\rangle$};
		\draw[fill] (v1)  circle (2pt) node [above] {$\sfD^\mathrm{b}(\PP^1_k)$};

		\path[-] (v0) edge  node [above] {} (va);
		\path[-] (v0) edge  node [above] {} (vb);
		\path[-] (v0) edge  node [above] {} (vc);
		\path[-] (va) edge  node [above] {} (v1);
		\path[-] (vb) edge  node [above] {} (v1);
		\path[-] (vc) edge  node [above] {} (v1);

  \end{tikzpicture}
\end{displaymath}
where $\langle\mcO(i)\rangle$ is shorthand for $\thick(\mcO(i))$. The claim that $\Spcnt \sfK$ is empty then follows from Lemma~\ref{lem:serialkiller}.
\end{ex}

\begin{ex}\label{ex:bicycle}
Let $\Lambda$ be the path algebra of the $2$-cycle 
\[
\xymatrix{
1 \ar[rr]<0.75ex>^-{\alpha} \ar@{<-}[rr]<-0.75ex>_-{\beta} && 2
}
\]
modulo the square of the radical, so $\Lambda$ is a $4$-dimensional self-injective algebra with $2$-dimensional projective right modules $P_1$ and $P_2$ (also known as the preprojective algebra of type $A_2$). We consider $\sfK = \sfD^\mathrm{perf}(A)$. The lattice $T(\sfK)$ is 
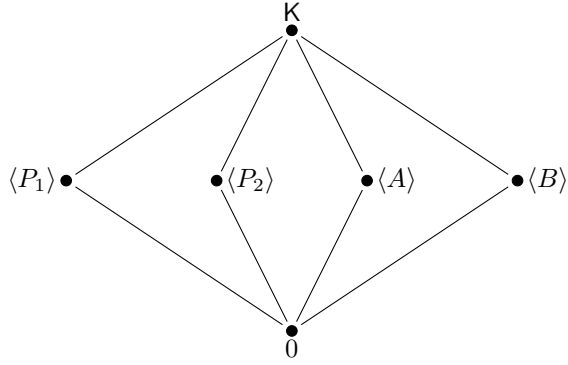
\begin{displaymath}
	\begin{tikzpicture}
    
		\node (v0) at (0,0) {};
    \node (va) at (-3,2) {};
    \node (vb) at (-1,2) {};
		\node (vc) at (1,2) {};
		\node (vd) at (3,2) {};
		\node (v1) at (0,4) {};
    
    \draw[fill] (v0)  circle (2pt) node [below] {0};
		\draw[fill] (va)  circle (2pt) node [left] {$\langle P_1 \rangle$};
		\draw[fill] (vb)  circle (2pt) node [right] {$\langle P_2 \rangle$};
		\draw[fill] (vc)  circle (2pt) node [right] {$\langle A \rangle$};
		\draw[fill] (vd)  circle (2pt) node [right] {$\langle B \rangle$};
		\draw[fill] (v1)  circle (2pt) node [above] {$\sfK$};

		\path[-] (v0) edge  node [above] {} (va);
		\path[-] (v0) edge  node [above] {} (vb);
		\path[-] (v0) edge  node [above] {} (vc);
		\path[-] (v0) edge  node [above] {} (vd);
		\path[-] (va) edge  node [above] {} (v1);
		\path[-] (vb) edge  node [above] {} (v1);
		\path[-] (vd) edge  node [above] {} (v1);
		\path[-] (vc) edge  node [above] {} (v1);

  \end{tikzpicture}
\end{displaymath}
where $A = P_1 \stackrel{\alpha}{\to} P_2$ and $B = P_2 \stackrel{\beta}{\to} P_1$. It again follows from Lemma~\ref{lem:serialkiller} that $\Spcnt \sfK = \varnothing$.
\end{ex}

These last two examples seem quite typical: when the lattice of thick subcategories has some `combinatorial component', i.e.\ a non-distributive piece, it often seems to be embedded such that $\Spcnt$ disintegrates. This is also the case, for instance, for the bounded derived category of an elliptic curve or for the bounded derived category of the path algebra of $A_n$ for $n\geq2$. However, one can dodge this issue and somehow separate the topological and combinatorial components, so to speak, in certain cases.

\begin{ex}\label{ex:serre1}
Consider again the case of $\sfK = \sfD^\mathrm{b}(\PP^1_k)$, which has Serre functor $S = (-)\otimes \Sigma \mcO(-2)$. We can consider the bounded sublattice $T^S(\sfK)$ of $\sfK$ consisting of those thick subcategories which are closed under $S$. One easily checks that $T^S(\sfK) = T^\otimes(\sfK)$ which is isomorphic to the coherent frame of Thomason subsets of $\PP^1_k$. Thus we recover $\PP^1_k$, up to Hochster duality, from the Serre functor invariant thick subcategories.
\end{ex}

\begin{ex}\label{ex:serre2}
Let $\Lambda$ be the algebra of Example~\ref{ex:bicycle} and $\sfK = \sfD^\mathrm{perf}(\Lambda)$. The Serre functor $S$ of $\sfK$ is given by the Nakayama automorphism which interchanges $P_1$ and $P_2$. An easy computation shows that $T^S(\sfK) \cong \Sierp$, i.e.\ the only Serre functor invariant thick subcategories are $0$ and $\sfK$. Thus $\Spec T^S(\sfK) \cong \ast$ consists of a single point.
\end{ex}

This example generalizes, for instance, to the situation of $\sfD^\mathrm{b}(kQ)$ for $Q$ a quiver of type $A_n$ or $D_n$ where the only Serre functor invariant thick subcategories are also $0$ and the whole category.

These examples are discussed a little further in Section~\ref{ssec:serre}.



\section{Our ignorance}\label{sec:questions}

There is a lot we don't know, and so we don't attempt to give a particularly exhaustive list of open problems.  Let us rather make a few further comments and pose a few concrete challenges (beyond the obvious, standard, and vexing challenge of performing meaningful computations in this setting).

\subsection{Realization}

It is natural to ask which lattices are of the form $T(\sfK)$ for an essentially small triangulated category $\sfK$.

\begin{qn}
Given a complete lattice $L$ is there a triangulated category $\sfK$ such that $T(\sfK)\cong L$?
\end{qn}

The answer to this question is already known in the nicest possible case.

\begin{prop}
Suppose that $F$ is a coherent frame. Then there is a commutative ring $R$ such that $F \cong T(\sfD^\mathrm{perf}(R))$.
\end{prop}
\begin{proof}
By \cite{Hochster} there is a commutative ring $R$ such that $\Spec R \cong (\Spec F)^\vee$. By Thomason's theorem \cite{Thomclass} we have $\Spc \sfD^\mathrm{perf}(R) \cong \Spec R$. As the perfect complexes are generated by the tensor unit $R$, this gives $F \cong T(\sfD^\mathrm{perf}(R))$. 
\end{proof}

However, we do not (to the authors' knowledge) know the answer for any more general class of lattices that isn't obtained from the above by some sleight of hand.

\subsection{Serre functor invariant thick subcategories}\label{ssec:serre}

We learn from tt-geometry that, for a tt-category $\sfK$, although $T(\sfK)$ might be non-distributive and resist our attempts to understand it, the lattice $T^\otimes(\sfK)$ is as nice as can be and is accessible via Stone duality. We are led to ask if there are other natural ways to produce sublattices of $T(\sfK)$ which are spatial frames?

We have seen in Examples~\ref{ex:serre1} and \ref{ex:serre2} that sometimes taking the lattice of Serre functor invariant thick subcategories $T^S(-)$ accomplishes this. This point of view is actually related to tt-geometry.

\begin{ex}
Suppose that $X$ is a smooth irreducible projective scheme. Then $\sfD^\mathrm{perf}(X)$ is a tt-category and has a Serre functor $S$ which can be described as $(-)\otimes \Sigma^{\dim X}\omega_X$ where $\omega_X$ is the canonical bundle. Thus, being a Serre functor invariant thick subcategory just means being closed under tensoring with powers of $\omega_X$.

If $\omega_X$ or $\omega_X^{-1}$ is ample this is nothing but being a tensor ideal. Thus in this setting $T^S(\sfK) = T^\otimes(\sfK)$ is a coherent frame, and we see that this lattice is actually intrinsic to $\sfD^\mathrm{perf}(X)$ and not dependent on the monoidal structure. This is closely related to, and gives another route to, the reconstruction theorem of Bondal and Orlov \cite{BOreconstruction}.
\end{ex}

It is not always the case for an essentially small triangulated category $\sfK$ with a Serre functor $S$ that $T^S(\sfK)$ is distributive. If $\sfK$ is Calabi-Yau, i.e.\ $S\cong \Sigma^n$ for some $n$, then $T^S(\sfK) = T(\sfK)$ and there are examples in which this lattice is non-distributive. For instance this happens for the bounded derived category of an ellptic curve and in discrete cluster categories of infinite type A \cite{GratzZvonareva}.

On the other hand small examples from the representation theory of finite dimensional algebras, e.g.\ Example~\ref{ex:serre2}, are quite suggestive and motivate the following question.

\begin{qn}
Let $\sfK$ be an essentially small triangulated category.
\begin{itemize}
\item[i)] If $\sfK$ has a Serre functor $S$ under which conditions is $T^S(\sfK)$ distributive?
\item[ii)] Can we identify, through other means, `interesting' distributive sublattices of $T(\sfK)$? For instance, are there interesting maximal bounded distributive sublattices?
\end{itemize}
\end{qn}

\begin{ex}
If $E$ is an elliptic curve over a field $k$ then the maximal distributive sublattices of $\sfD^\mathrm{b}(E)$ are all isomorphic to one of $\Sierp \times \Sierp$ or $T^\otimes(\sfD^\mathrm{b}(E))$. This seems meaningful.
\end{ex}



\bibliography{greg_bib}

\end{document}